\documentclass[11pt]{article}
\usepackage{latexsym,amsfonts,amssymb,amsmath,amsthm}
\usepackage{graphicx}
\usepackage{cite}
\usepackage[usenames,dvipsnames]{color}
\usepackage{ulem}
\usepackage[colorlinks=true]{hyperref}
\hypersetup{urlcolor=blue, citecolor=red}

\newcommand{\norm}[1]{||#1||}

\newcommand{\ucite}[1]{\cite{#1}}
\begin{document}
\setlength{\baselineskip}{16pt}

\parindent 0.5cm
\evensidemargin 0cm \oddsidemargin 0cm \topmargin 0cm \textheight
22cm \textwidth 16cm \footskip 2cm \headsep 0cm

\newtheorem{theorem}{Theorem}[section]
\newtheorem{lemma}[theorem]{Lemma}
\newtheorem{proposition}[theorem]{Proposition}
\newtheorem{definition}{Definition}[section]
\newtheorem{example}{Example}[section]
\newtheorem{corollary}[theorem]{Corollary}

\newtheorem{remark}{Remark}[section]
\newtheorem{property}[theorem]{Property}
\numberwithin{equation}{section}
\newtheorem{mainthm}{Theorem}
\newtheorem{mainlem}{Lemma}

\numberwithin{equation}{section}

\def\p{\partial}
\def\I{\textit}
\def\R{\mathbb R}
\def\C{\mathbb C}
\def\u{\underline}
\def\l{\lambda}
\def\a{\alpha}
\def\O{\Omega}
\def\e{\epsilon}
\def\ls{\lambda^*}
\def\D{\displaystyle}
\def\wyx{ \frac{w(y,t)}{w(x,t)}}
\def\imp{\Rightarrow}
\def\tE{\tilde E}
\def\tX{\tilde X}
\def\tH{\tilde H}
\def\tu{\tilde u}
\def\d{\mathcal D}
\def\aa{\mathcal A}
\def\DH{\mathcal D(\tH)}
\def\bE{\bar E}
\def\bH{\bar H}
\def\M{\mathcal M}
\renewcommand{\labelenumi}{(\arabic{enumi})}

\def\disp{\displaystyle}
\def\undertex#1{$\underline{\hbox{#1}}$}
\def\card{\mathop{\hbox{card}}}
\def\sgn{\mathop{\hbox{sgn}}}
\def\exp{\mathop{\hbox{exp}}}
\def\OFP{(\Omega,{\cal F},\PP)}
\newcommand\JM{Mierczy\'nski}
\newcommand\RR{\ensuremath{\mathbb{R}}}
\newcommand\CC{\ensuremath{\mathbb{C}}}
\newcommand\QQ{\ensuremath{\mathbb{Q}}}
\newcommand\ZZ{\ensuremath{\mathbb{Z}}}
\newcommand\NN{\ensuremath{\mathbb{N}}}
\newcommand\PP{\ensuremath{\mathbb{P}}}
\newcommand\abs[1]{\ensuremath{\lvert#1\rvert}}

\newcommand\normf[1]{\ensuremath{\lVert#1\rVert_{f}}}
\newcommand\normfRb[1]{\ensuremath{\lVert#1\rVert_{f,R_b}}}
\newcommand\normfRbone[1]{\ensuremath{\lVert#1\rVert_{f, R_{b_1}}}}
\newcommand\normfRbtwo[1]{\ensuremath{\lVert#1\rVert_{f,R_{b_2}}}}
\newcommand\normtwo[1]{\ensuremath{\lVert#1\rVert_{2}}}
\newcommand\norminfty[1]{\ensuremath{\lVert#1\rVert_{\infty}}}

\title{Structure of $\omega$-Limit Sets for Almost-Periodic Parabolic Equations on $S^1$ with Reflection Symmetry}

\author {
\\
Wenxian Shen\\
Department of Mathematics and Statistics\\
 Auburn University, Auburn, AL 36849, USA
\\
\\
Yi Wang\thanks{Partially supported by NSF of China No.11371338, 11471305, Wu Wen-Tsun Key Laboratory and the Fundamental
Research Funds for the Central Universities.}  $\,\,$ and $\,$ Dun Zhou\thanks{Partially supported by NSF of China No.11601498.}\\
Department of Mathematics\\
 University of Science and Technology of China
\\ Hefei, Anhui, 230026, P. R. China
\\
\\
}
\date{}

\maketitle
% insert the table of contents
%\tableofcontents

%---------------------SECTION DIVIDE LINE---------------------------
\begin{abstract}
The structure of the $\omega$-limit sets is thoroughly investigated for the skew-product semiflow which is generated by a scalar reaction-diffusion equation
\begin{equation*}
u_{t}=u_{xx}+f(t,u,u_{x}),\,\,t>0,\,x\in S^{1}=\mathbb{R}/2\pi \mathbb{Z},
\end{equation*}
where $f$ is uniformly almost periodic in $t$ and satisfies $f(t,u,u_x)=f(t,u,-u_x)$. We show that any $\omega$-limit set $\Omega$ contains at most two minimal sets. Moreover, any hyperbolic $\omega$-limit set $\Omega$  is a spatially-homogeneous $1$-cover of hull $H(f)$.  When $\dim V^c(\Omega)=1$ ($V^c(\O)$ is the center space associated with $\O$), it is proved that either $\Omega$ is a spatially-homogeneous, or $\O$ is a spatially-inhomogeneous $1$-cover of $H(f)$.
\end{abstract}

\section{Introduction}

In this paper we investigate the dynamics of bounded solutions for the scalar reaction-diffusion equations on the circle
\begin{equation}\label{equation-1}
u_{t}=u_{xx}+f(t,u,u_{x}),\,\,t>0,\,x\in S^{1}=\mathbb{R}/2\pi \mathbb{Z},
\end{equation}
where $f:\RR\times\RR\times\RR\to \RR$ together with all its derivatives (up to order $2$) are almost periodic in $t$ uniformly for $(u,p)$ in compact subsets.

To carry out our investigation, we embed \eqref{equation-1} into a skew-product semiflow in the following way. Let $f_{\tau}(t,u,p)=f(t+\tau,u,p)(\tau \in \RR)$ be the time-translation of $f$, then the function $f$ generates a family
$\{f_{\tau}|\tau \in \mathbb{R}\}$ in the space of continuous functions $C(\mathbb{R}\times \mathbb{R} \times \mathbb{R},\mathbb{R})$ equipped with the compact open topology. By the Ascoli--Arzela theorem, the hull $H(f)$ of $f$ (i.e. the closure of $\{f_{\tau}|\tau\in \mathbb{R}\}$ in the compact open topology) is a compact metric space and every $g\in H(f)$ is uniformly almost periodic with the same regularity as $f$.  Hence, the action of time-translation $g\cdot t\equiv g_{t}\,(g\in H(f))$ defines a compact minimal flow on $H(f)$ (cf. \cite{Sell,Shen1998}). This means that $H(f)$ is the only nonempty compact subset of itself
that is invariant under the flow $g\cdot t$. As a consequence, equation \eqref{equation-1}
induces a family of equations associated to each $g\in H(f)$,
\begin{equation}\label{equation-lim1}
u_{t}=u_{xx}+g(t,u,u_{x}),\,\,\quad t\geq 0,\quad x\in S^{1}.
\end{equation}

Let $X$ be the fractional power space associated with the operator $u\rightarrow
-u_{xx}:H^{2}(S^{1})\rightarrow L^{2}(S^{1})$ such that the embedding
relation $X\hookrightarrow C^{1}(S^{1})$ is satisfied. Choose any $u\in X$, then \eqref{equation-lim1} admits (locally) a unique solution $\varphi(t,\cdot;u,g)$ in $X$ with $\varphi(0,\cdot,u,g)=u(\cdot)$. Moreover, the solution continuously depends on $g\in H(f)$ and $u\in
X$. Therefore, \eqref{equation-lim1} defines a (local) skew product semiflow $\Pi^{t}$ on $X\times
H(f)$:
\begin{equation}\label{equation-lim2}
\Pi^{t}(u,g)=(\varphi(t,\cdot;u,g),g\cdot t),\quad t\geq 0.
\end{equation}
By the standard a priori estimates for parabolic equations (see \cite{Hen}), if
$\varphi(t,\cdot;u,g) (u\in X)$ is bounded in $X$ in the existence interval of the solution, then $u$ is a globally defined classical solution. The study of the asymptotic behavior of $\varphi(t,\cdot;u,g)$ boils down to  the problem of understanding the structure of the $\omega$-limit set $\omega(u,g)$ of $\varphi(t,\cdot;u,g)$. Note that, for any $\delta>0$, $\{\varphi(t,\cdot;u,g):t\ge \delta\}$ is relatively compact in $X$. Consequently, $\omega(u,g)$ is a nonempty connected compact subset of $X\times H(f)$. It is also known that $\Pi^{t}$ restricted on the $\omega$-limit set $\omega(u,g)$ has a unique continuous backward time extension, and hence, admits a flow extension (see, e.g. \cite{hale1988asymptotic}).

When $f$ is independent of $t$,  Massatt \cite{Massatt1986} and Matano \cite{Matano} showed independently that any periodic orbit is a rotating wave $u=\phi(x-ct)$ for some $2\pi$-periodic function $\phi$ and constant $c$, and  any $\omega$-limit set is either itself a single rotating wave, or a set of equilibria differing by phase shift in $x$. Moreover, if  $f(u,u_x)=f(u,-u_x)$, then any $\omega$-limit set is just an equilibrium.  In \cite{Fiedler},
the well-known Poincar\'{e}-Bendixson type theorem has been established by Fiedler and Mallet-Paret for $f$ which is independent of $t$ but can depend on $x$.

In the case that $f$ in \eqref{equation-1} is time-periodic with period $1$ (equivalently, $H(f)$ is homeomorphic to the circle $\mathcal{T}^1=\mathbb{R}/\mathbb{Z}$), Sandstede and Fiedler \cite{SF1} showed that the  $\omega$-limit set $\omega(u)$ can be viewed as a subset of the two-dimensional torus $\mathcal{T}^1\times S^1$ carrying a linear flow. Moreover, for $f(t,u,u_x)=f(t,u,-u_x)$, Chen and Matano \cite{Chen1989160} proved that the $\omega$-limit set $\omega(u)$ of any bounded solution consists of a unique time-periodic orbit with period $1$. For the general $f$ which is time-periodic but depending on
$x$, one may refer to Tere\v{s}\v{c}\'{a}k's work \cite{Te} to see that any $\omega$-limit set of the associated Poincar\'{e} map can be imbedded into a $2$-dimensional plane; and moreover, the chaotic behavior was exhibited in \cite{SF1} for such general $f$.

In the case that $f$ in \eqref{equation-1} is time almost-periodic, the present authors \cite{SWZ} thoroughly investigated  the structure of the minimal sets generated by \eqref{equation-1}. Among others, they have proved that a minimal set $M$ is a spatially-homogeneous $1$-cover of $H(f)$ if it is hyperbolic (i.e., the dimension of the center space $V^c(M)$ associated with $M$ is $0$). When ${\rm dim}V^c(M)=1$ (resp. ${\rm dim}V^c(M)=2$ with ${\rm dim}V^u(M)$ being odd), they have obtained that either $M$ is spatially-homogeneous; or $M$ can be (resp. residually) embedded into an almost-periodically (resp. almost-automorphically) forced circle-flow $S^1\times H(f)$.
The structure of the minimal sets with higher-dimensional center spaces still remains open; and moreover, that of the $\omega$-limit sets for
\eqref{equation-1} have not been studied so far from this point of view.

 In the present paper we will focus on the structures of the $\omega$-limit sets for \eqref{equation-1} under the assumption of the reflection symmetry $f(t,u,u_x)=f(t,u,-u_x)$. For such circumstance, it is already known that any minimal set is an almost $1$-cover of $H(f)$ (see \cite{SWZ}). Note that a minimal set is just a special case of the $\omega$-limit sets. The current paper is devoted to a basic description of the structure of a general $\omega$-limit set of \eqref{equation-lim2}. More precisely, assume that $f(t,u,u_x)=f(t,u,-u_x)$ and let $(u_0,g_0)\in X\times H(f)$ be such that the motion $\Pi^{t}(u_0,g_0)$($t\ge 0$) is bounded. Then we have the following three main theorems (We write $\Omega$ for the $\omega$-limit set $\omega(u_0,g_0)$):

\bigskip

\noindent Theorem A. {\it
The $\omega$-limit set $\Omega$ contains at most two (obviously at least one) minimal sets. (See also Theorem \ref{structure}).}

\vskip 3mm
\noindent Theorem B. {\it
Any hyperbolic $\omega$-limit set $\Omega$ is a spatially-homogeneous  $1$-cover of $H(f)$. (See also Theorem \ref{hyperbolic0}).}

\vskip 3mm
\noindent Theorem C. {\it
 Assume that $\dim V^c(\Omega)=1$. Then $\Omega$ is either spatially-homogeneous or a spatially-inhomogeneous $1$-cover of $H(f)$. (See also Theorem \ref{norma-hyper}).
}
\vskip 3mm

All these structural theorems for the $\omega$-limit sets $\Omega$ of \eqref{equation-lim2} are natural generalizations from the autonomous and time-periodic cases (\cite{Massatt1986,Matano,Chen1989160,SF1}) to the time almost-periodic systems. One may further find that Theorems B\&C have ever been considered in \cite{SWZ} under the additional assumption that $\Omega$ itself is minimal. As a consequence, the structural theorems here are also extensions of \cite{SWZ} from the minimal sets to general $\omega$-limit sets for \eqref{equation-lim2}.

For separated boundary conditions, one may observe the related structural theorems for the $\omega$-limit sets, or even the $1$-covering property of the hyperbolic $\omega$-limit sets with general nonlinearity $f(t,x,u,u_x)$ (see Shen-Yi\ucite{Shen1995114,ShenYi-2,ShenYi-JDDE96,Shen1998}). However, with periodic boundary condition one can not expect the corresponding phenomena of the $\omega$-limit sets in general (see \cite{Po1,SF1,SWZ} and references therein). For instance, for system \eqref{equation-1} with $f$ being replaced by $f(x,u,u_x)$, Fiedler and Mallet-Paret \cite{Fiedler} proved that any $\omega$-limit set of \eqref{equation-1} is either a single periodic orbit or it consists of equilibria and connecting (homoclinic and heteroclinic) orbits. When $f$ is replaced by $f(t,x,u,u_x)$, it is only known that the uniformly stable minimal set is topologically
conjugate to a skew-product flow on some $\hat M\subset\RR^2\times H(f)$ (see \cite[Theorem 5.2]{SWZ}). As a consequence, the understanding of the structure of general $\omega$-limit sets of \eqref{equation-1} with $f=f(t,x,u,u_x)$ is very limited.

Theorems A and B here help us understand the structures of the $\omega$-limit sets under the assumption of the reflection-symmetry $f(t,u,u_x)=f(t,u,-u_x)$. Moreover, to the best of our knowledge, the structure of general $\omega$-limit sets with nontrivial center spaces have been hardly studied for both separated boundary and periodic boundary conditions (see some related results for the minimal sets in \cite{ShenYi-JDDE96, SWZ} and for tridiagonal competitive-cooperative systems \cite{Fangchun2013,Fangchun2015}). Theorem C seems to be the first attempt to tackle this problem for the general $\omega$-limit sets with nontrivial center spaces.

This paper is organized as follows. In section 2, we introduce some notations, relevant
definitions and preliminary results, including the almost periodic (automorphic) theory and the invariant manifolds theory for skew-product semiflows, which will be important to our proofs. In section 3, we will prove Theorem A (see the detail in Theorem \ref{structure}) for the $\omega$-limit set $\Omega$. In section 4, we will investigate the structure of the hyperbolic $\omega$-limit sets and prove Theorem B (see also Theorem \ref{hyperbolic0}). In section 5, we will focus on Theorem C (see also Theorem \ref{norma-hyper}). Finally, we will give several examples in Section 6 related to Theorem \ref{structure} and
Theorem \ref{norma-hyper}. These examples illustrate that $\Omega$ can possess one or two minimal sets (i.e., Theorem \ref{structure}(ii)-(iii) indeed may occur); and moreover, if $\dim V^c(\Omega)\geq 2$ then the conclusion of Theorem 5.1 cannot hold anymore.

\section{Preliminaries}
In this section, we summarize some preliminary materials to be used in later sections. We start by introducing some basic concepts of skew-product semiflows, etc.  Next, we recall the definitions of almost periodic (automorphic) functions  and some of their properties.  We then present some basic properties of zero numbers of solutions for linear parabolic equations. Finally, we introduce the invariant subspaces and invariant manifold theory for the linear and nonlinear parabolic equations on $S^1$.

\subsection{Basic concepts}

 Let $Y$ be a compact metric space with metric $d_{Y}$, and
$\sigma:Y\times \RR\to Y, (y,t)\mapsto y\cdot t$ be a continuous
flow on $Y$, denoted by $(Y,\sigma)$ or $(Y,\RR)$. A subset
$S\subset Y$ is {\it invariant} if $\sigma_t(S)=S$ for every $t\in
\RR$. A subset $S\subset Y$ is called {\it minimal} if it is
compact, invariant and the only non-empty compact invariant subset
of it is itself.  Every compact and $\sigma$-invariant set contains a minimal subset and a subset $S$ is minimal if and only if every trajectory is dense
in $S$. The continuous flow $(Y,\sigma)$ is called to be
{\it recurrent} or {\it minimal} if $Y$ is minimal. A pair
$y_1,y_2$ of different elements of $Y$ is {\it distal}, if there is a $\delta>0$ such
that $d_Y(y_1\cdot t,y_2\cdot t)>\delta$ for every $t\in
\mathbb{R}$, the pair $y_1,y_2$ is called {\it proximal} if it not distal. We
say that the flow $(Y,\sigma)$ is distal when, each pair
$y_1,y_2$ of different elements of $Y$ is distal.

Hereafter, we always assume that $Y$ is minimal and distal.

 Let $X,Y$ be metric spaces and $(Y,\sigma)$ be a compact flow (called the base flow). Let also
 $\RR^+=\{t\in \RR:t\ge 0\}$. A
 skew-product semiflow $\Pi^t:X\times Y\rightarrow X\times Y$ is a semiflow of the following form
 \begin{equation}\label{skew-product-semiflow}
 \Pi^{t}(u,y)=(\varphi(t,u,y),y\cdot t),\quad t\geq 0,\, (u,y)\in X\times Y,
 \end{equation}
satisfying (i) $\Pi^{0}={\rm Id}_X$ and (ii) the co-cycle property:
$\varphi(t+s,u,y)=\varphi(s,\varphi(t,u,y),y\cdot t)$ for each $(u,y)\in X\times Y$ and $s,t\in \RR^+$.
 A subset $A\subset
X\times Y$ is {\it positively invariant} if $\Pi^t(A)\subset A$ for
all $t\in \RR^+$. The {\it forward orbit} of any $(u,y)\in X\times
Y$ is defined by $\mathcal{O}^+(u,y)=\{\Pi^t(u,y):t\ge 0\}$, and the
{\it $\omega$-limit set} of $(u,y)$ is defined by
$\omega(u,y)=\{(\hat{u},\hat{y})\in X\times Y:\Pi^{t_n}(u,y)\to
(\hat{u},\hat{y}) (n\to \infty) \textnormal{ for some sequence
}t_n\to \infty\}$.

A {\it flow extension} of a skew-product semiflow $\Pi^t$
 is a continuous skew-product flow $\hat{\Pi}^t$ such that $\hat{\Pi}^t(u,y)=\Pi^t(u,y)$ for each
$(u,y)\in X\times Y$ and $t\in \RR^+$. A compact positively invariant subset is said to admit {\it a flow extension} if the
semiflow restricted to it does. Actually, a compact positively invariant set $K\subset X\times Y$ admits a flow extension if every
point in $K$ admits a unique backward orbit which remains inside the
set $K$ (see \cite[part II]{Shen1998}). In particular, the $\omega$-limit set $\omega(u,g)$ defined above admits a flow extension (see, e.g. \cite{hale1988asymptotic,Shen1998}). A compact invariant set $K\subset X\times Y$ for $\Pi^t$ ($\Pi^{t}(K)\subset K$ for all $t\in\mathbb{R}$) is called {\it minimal} if it
does not contain any other nonempty compact invariant set
than itself.

Let $K\subset X\times Y$ be a positively invariant set for $\Pi^t$ which admits a flow extension. For each $(u,y)\in K$ one can define a backward orbit of $(u,y)$ as $\mathcal{O}^-(u,y)=\{\Pi^t(u,y):t\le 0\}$ and the $\alpha$-limit set of $(u,y)$ is defined as $\alpha(u,y)=\{(\hat{u},\hat{y})\in X\times Y:\Pi^{t_n}(u,y)\to
(\hat{u},\hat{y}) (n\to \infty) \textnormal{ for some sequence
}t_n\to -\infty\}$.
Let also $p:X\times Y\to Y$ be the natural projection. Then $p$ is a flow
homomorphism for the flows $(K,\RR)$ and $(Y,\sigma)$.  Moreover, $K\subset X\times Y$ is called an
{\it almost $1$-cover} ({\it $1$-cover}) of $Y$ if card$(p^{-1}(y)\cap K)=1$
 for at least one $y\in Y$ (for any $y\in Y$).

\subsection{Almost-periodic and almost-automorphic functions }

Let $D$ be a subset of $\RR^m$. We list the following definitions and notations in this subsection.

\begin{definition}\label{admissible}
{\rm A continuous function $f:\RR\times D\to\RR, (t,w)\mapsto f(t,w),$ is said to be {\it
admissible} if $f(t,w)$ is bounded and uniformly continuous on
$\RR\times K$ for any compact subset $K\subset D$. $f$ is $C^r$ ($r\ge 1$) {\it admissible} if $f$ is $C^r$ in $w\in D$ and Lipschitz in $t$, and $f$ as well as its partial derivatives to order $r$ are admissible.}
\end{definition}

Let $f\in C(\RR\times D,\RR) (D\subset \RR^m)$ be admissible. Then
$H(f)={\rm cl}\{f\cdot\tau:\tau\in \RR\}$ is called the {\it hull of
$f$}, where $f\cdot\tau(t,\cdot)=f(t+\tau,\cdot)$ and the closure is
taken under the compact open topology. Moreover, $H(f)$ is compact
and metrizable under the compact open topology (see \cite{Sell,Shen1998}).
The time translation $g\cdot t$ of $g\in H(f)$ induces a natural
flow on $H(f)$ (cf. \cite{Sell}). Moreover, if $f$ is $C^r$-admissible then, $g$ is also $C^r$-admissible for any $g\in H(f)$ (see \cite{Shen1998}).
\begin{definition}{\rm
A function $f\in C(\RR,\RR)$ is {\it almost periodic} if, for any
$\varepsilon>0$, the set
$T(\varepsilon):=\{\tau:\abs{f(t+\tau)-f(t)}<\varepsilon,\,\forall
t\in \RR\}$ is relatively dense in $\RR$. $f$ is {\it almost automorphic} if for every $\{t'_k\}\subset\mathbb{R}$ there is a subsequence $\{t_k\}$
and a function $g:\mathbb{R}\to \mathbb{R}$ such that $f(t+t_k)\to g(t)$ and $g(t-t_k)\to f(t)$ point wise. A function $f\in C(\RR\times D,\RR)(D\subset \RR^m)$ is {\it uniformly almost periodic (automorphic)} {\it in $t$}, if $f$ is both admissible and, for each fixed $d\in D$, $f(t,d)$ is almost periodic (automorphic) with respect to $t\in \RR$.}

\end{definition}

\begin{remark}\label{a-p-to-minial}
{\rm  If $f$ is a uniformly almost automorphic function in $t$, then $H(f)$ is always {\it minimal}, and there is a residual set $Y'\subset H(f)$, such that all $g\in Y'$ is an almost automorphic function in $t$. If $f$ is a uniformly almost perioidc function in $t$, then $H(f)$ is always {\it minimal and distal}, and every $g\in H(f)$ is uniformly almost periodic function (see, e.g. \cite{Shen1998}).}
\end{remark}

\subsection{Zero number function}

Given a $C^{1}$-smooth function $u:S^{1}\rightarrow \mathbb{R}^{1}$, the zero number of $u$ is
 defined as
$$z(u(\cdot))={\rm card}\{x\in S^{1}|u(x)=0\}.$$
The following lemma  was originally presented in \cite{2038390} and \cite{H.MATANO:1982} and was improved in \cite{Chen98}.
\begin{lemma}\label{zero-number}
Consider the linear system
\begin{equation}\label{linear-equation}
\begin{cases}
\varphi_{t}=a(t,x)\varphi_{xx}+b(t,x) \varphi_{x}+c(t,x)\varphi,\quad x\in S^1,\\
\varphi_{0}=\varphi(0,\cdot)\in H^{1}(S^{1}),
\end{cases}
\end{equation}
where $a,a_t,a_x,b$ and $c$ are bounded continuous functions, $a\ge \delta >0$. Let $\varphi(t,x)$ be a classical nontrivial solution of \eqref{linear-equation}.
Then the following properties holds.\par
{\rm (a)} $z(\varphi(t,\cdot))<\infty,\forall t>0$ and is non-increasing in t.\par
{\rm (b)}  $z(\varphi(t,\cdot))$ can drop only at $t_{0}$ such that $\varphi(t_{0},\cdot)$ has a
multiple zero in $S^{1}$.\par
{\rm (c)}  $z(\varphi(t,\cdot))$ can drop only finite many times,and there exists a $T>0$
such that $\varphi(t,\cdot)$ has only simple zeros in $S^{1}$ as $t\geq T$(hence
$z(\varphi(t,\cdot))=constant$ as $t\geq T$).
\end{lemma}

\begin{lemma}\label{difference-lapnumber}
For any $g\in H(f)$, Let $\varphi(t,\cdot;u,g)$ and $\varphi(t,\cdot;\hat{u},g)$ be two
distinct solutions of {\rm (\ref{equation-lim1})} on
$ \mathbb{R}^+$. Then

\begin{description} \item[{\rm (a)}]
$z(\varphi(t,\cdot;u,g)-\varphi(t,\cdot;\hat{u},g))<\infty$ for $t>0$ and is non-increasing in t;

\item[{\rm (b)}] $z(\varphi(t,\cdot;u,g))-\varphi(t,\cdot;\hat{u},g)))$
strictly decreases at $t_0$ such that the function $\varphi(t_0,\cdot;u,g))-\varphi(t_0,\cdot;\hat{u},g)$ has a multiple
zero in $S^1$;

\item[{\rm (c)}] $z(\varphi(t,\cdot;u,g))-\varphi(t,\cdot;\hat{u},g)))$ can drop only finite many times, and there exists a $T>0$ such that  $$z(\varphi(t,\cdot;u,g))-\varphi(t,\cdot;\hat{u},g)))\equiv
\textnormal{constant}$$ for all $t\ge T$.
\end{description}
\end{lemma}
\begin{proof}
See Lemma 2.3 in \cite{SWZ}.
\end{proof}

 \begin{lemma}\label{zero-cons-local}
 Let $u\in X$ be such that $u$ has only simple zeros on $S^1$, then there is $\delta>0$ such that for any $v\in X$ with $\|v\|<\delta$, one has
  \begin{equation*}
    z(u)=z(u+v)
  \end{equation*}
\end{lemma}
\begin{proof}
  See Corollary 2.1 in \cite{SF1} or Lemma 2.3 in \cite{Chen1989160}.
\end{proof}
The proof of the following lemma can be found in \cite[Lemma 2.4]{SWZ}.
\begin{lemma}\label{sequence-limit}
Fix $g,\ g_{0}\in H(f)$. Let $(u^{i},g)\in p^{-1}(g),(u_{0}^{i},g_{0})\in p^{-1}(g_{0}) (i=1,\ 2,\ u^{1}\neq u^{2},\ u_{0}^{1}\neq u_{0}^{2})$ be such that $\Pi^{t}(u^{i},g)$ is defined on $\mathbb{R}^{+}$ and $\Pi^{t}(u_{0}^{i},g_{0})$ is defined on $\mathbb{R}$. If there exists a sequence $t_{n}\rightarrow +\infty$ ({\rm resp}. $s_{n}\rightarrow -\infty$) as $n\rightarrow \infty$, such that $\Pi^{t_{n}}(u^{i},g)\rightarrow (u_{0}^{i},g_{0})$ ({\rm resp}. $\Pi^{s_{n}}(u^{i},g)\rightarrow (u_{0}^{i},g_{0})$) as $n\rightarrow \infty (i=1,2)$, then
$$z(\varphi(t,\cdot;u_{0}^{1},g_{0})-\varphi(t,\cdot;u_{0}^{2},g_{0}))\equiv \textnormal{constant},$$
for all $t\in \mathbb{R}$.
\end{lemma}

\subsection{Invariant subspaces for linear parabolic equations on $S^1$}

Consider the following linear parabolic equation:
\begin{equation}\label{linear-equation2}
\psi_t=\psi_{xx}+a(x,\omega\cdot t)\psi_x+b(x,\omega\cdot t)\psi,\,\,t>0,\,x\in S^{1}=\mathbb{R}/2\pi \mathbb{Z},
\end{equation}
where $\omega\in \Omega$, $\omega\cdot t$ is a flow on a compact  metric space $\Omega$,  $a,b: S^1\times \Omega\rightarrow \mathbb{R}^1$ are continuous, and $a^\omega(t,x)=a(x,\omega\cdot t)$, $b^\omega(t,x)=b(x,\omega\cdot t)$ are H\"older continuous in $t$.

Throughout this subsection, $X$ is the fractional power space associated with the operator $u\rightarrow
-u_{xx}:H^{2}(S^{1})\rightarrow L^{2}(S^{1})$ such that the embedding
relation $X\hookrightarrow C^{1}(S^{1})$ is satisfied.

Now we recall the concept of exponential dichotomy (ED) and Sacker-Sell spectrum.\par
Let $\Psi(t,\omega):X\rightarrow X$ be the evolution operator generated by \eqref{linear-equation2}, that is, the evolution operator of the following equation:
\begin{equation}\label{linear-opera}
  v'=A(\omega\cdot t)v,\quad t>0,\,\omega\in \Omega,\, v\in X,
\end{equation}
where $A(\omega)v=v_{xx}+a(x,\omega)v_x+b(x,\omega)v$, and $\omega\cdot t$ is as in \eqref{linear-equation2}.

Define $\bar\Pi^t: X\times\Omega\rightarrow X\times\Omega$ by
\begin{equation}\label{exponential-dicho}
  \bar\Pi^t(v,\omega)=(\Psi(t,\omega)v,\omega\cdot t),
\end{equation}
Clearly, $\bar\Pi^t$ is a linear skew-product semiflow on $X\times\Omega$. We say $\Psi$ admits an {\it exponential dichotomy over} $\Omega$ if there exist $K>0$, $\alpha>0$ and continuous projections $P(\omega):X\rightarrow X$ such that for all $\omega \in \Omega$, $\Psi(t,\omega)|_{R(P(\omega))}:R(P(\omega))\rightarrow R(P(\omega\cdot t))$ is an isomorphism satisfying $\Psi(t,\omega)P(\omega)=P(\omega \cdot t)\Psi(t,\omega)$, $t\in \mathbb{R}^+$ (hence $\Psi(-t,\omega):=\Psi^{-1}(t,\omega\cdot-t):R(P(\omega))\to R(P(\omega\cdot-t))$ is well defined for $t\in\mathbb{R}^+$); moreover,
\begin{equation}\label{expone-dichot}
\begin{split}
\|\Psi(t,\omega)(I-P(\omega))\|\leq Ke^{-\alpha t},\quad t\geq 0,\\
\|\Psi(t,\omega)P(\omega)\|\leq Ke^{\alpha t},\quad t\leq 0.
\end{split}
\end{equation}
Here $R(P(\omega))$ is the range of $P(\omega)$. Let $\lambda\in \mathbb{R}$ and define $\bar\Pi^t_{\lambda}: X\times\Omega\rightarrow X\times\Omega$ by
\begin{equation*}
  \bar\Pi^t_{\lambda}(v,\omega)=(\Psi_{\lambda}(t,\omega)v,\omega\cdot t),
\end{equation*}
where $\Psi_{\lambda}(t,\omega)=e^{-\lambda t}\Psi(t,\omega)$. So, $\bar\Pi^t_{\lambda}$ is also a linear skew-product semiflow on
$X\times \Omega$. We call
$$
\sigma(\Omega)=\{\lambda\in \mathbb{R}:\bar\Pi^t_{\lambda}\ \text{has no exponential dichotomy over }\Omega\}
$$
the {\it Sacker-Sell spectrum} of \eqref{linear-equation2} or \eqref{linear-opera}. If $\Omega$ is  {\it compact and connected}, then the Sacker-Sell spectrum $\sigma(\Omega)=\cup_{k=0}^\infty I_k$, where $I_k=[a_k,b_k]$ and $\{I_k\}$ is ordered from right to left, that is, $\cdots<a_k\leq b_k<a_{k-1}\leq b_{k-1}<\cdots<a_0\leq b_0$ (see \cite{Chow1994,Sacker1978,Sacker1991}).\par
Given any given $0\leq n_1\leq n_2\leq\infty$. When $n_2\neq \infty$, let
\begin{equation}\label{twoside-estimate}
\begin{split}
V^{n_1,n_2}(\omega)=\{v\in X:&\|\Psi(t,\omega)v\|=o(e^{a^-t})\ \text{as}\ t\rightarrow -\infty\\ & \|\Psi(t,\omega)v\|=o(e^{b^+t})\ \text{as}\ t\rightarrow \infty\}
\end{split}
\end{equation}
where $a^-$, $b^+$ are such that $b_{n_2+1}<a^-<a_{n_2}\leq b_{n_1}<b^+<a_{n_1-1}$. Here $a_{n_1-1}=\infty$ if $n_1=0$. For any $a^-\in (b_{n_2+1},a_{n_2})$ (resp. $b^+\in (b_{n_1},a_{n_1-1})$), $\Psi_{a^-}(t,\omega)$ (resp. $\Psi_{b^-}(t,\omega)$) admits an exponential dichotomy over $\Omega$. Let $P_{a^-}(\omega)$ (resp. $P_{b^+}(\omega)$) be the associated continuous projection, which is in fact independent of $a^-\in (b_{n_2+1},a_{n_2})$ (resp. $b^+\in (b_{n_1},a_{n_1-1})$). Then $V^{n_1,n_2}(\omega)=R(I-P_{b^+}(\omega))\cap R(P_{a^-}(\omega))$. Recall that $\Psi(t,\omega)|_{R(P_{a^-}(\omega))}$ is an isomorphism, it then yields that $\Psi(t,\omega)v$ is well defined for $v\in V^{n_1,n_2}(\omega)$ and $t<0$. When $n_1<n_2=\infty$, let
\[
V^{n_1,\infty}(\omega)=\{v\in X :\|\Psi(t,\omega)v\|=o(e^{b^+t})\text{as }t\to\infty\}
\]
where $b^+$ is such that $b_{n_1}<b^+<\lambda$ for any $\lambda\in \cup_{k=0}^{n_1-1}I_k$.
$V^{n_1,n_2}(\omega)$ is called the {\it invariant subspace} of \eqref{linear-equation2} or \eqref{linear-opera} associated with the spectrum set $\cup_{k=n_1}^{n_2}I_k$ at $\omega\in \Omega$.\par

Suppose that $0\in \sigma(\Omega)$ and $n_0$ is such that $0\in I_{n_0}\subset\sigma(\Omega)$. Then $V^s(\omega)=V^{n_0+1,\infty}(\omega)$, $V^{cs}(\omega)=V^{n_0,\infty}(\omega)$, $V^{c}(\omega)=V^{n_0,n_0}(\omega)$, $V^{cu}(\omega)=V^{0,n_0}(\omega)$, and $V^u(\omega)=V^{0,n_0-1}(\omega)$ are referred to as {\it stable, center stable, center, center unstable}, and {\it unstable subspaces} of \eqref{linear-equation2} at $\omega\in \Omega$, respectively.

Suppose that $0\not\in \sigma(\Omega)$ and $n_0$ is such that $I_{n_0}\subset (0,\infty)$ and $I_{n_0+1}\subset (-\infty,0)$.
$V^s(\omega)=V^{n_0+1,\infty}(\omega)$ and $V^u(\omega)=V^{0,n_0}(\omega)$ are referred to as {\it stable} and {\it unstable subspaces} of \eqref{linear-equation2} at $\omega\in \Omega$, respectively.

\subsection{Invariant manifolds  of nonlinear parabolic equations on $S^1$}\label{invari-mani0}

Consider
\begin{equation}\label{lipschitz}
v'=A(\omega \cdot t)v+F(v,\omega \cdot t),\quad t>0,\ \omega\in \Omega,\ v\in X
\end{equation}
where $\omega\cdot t$ and $A(\omega\cdot t)$ are as in \eqref{linear-opera}, $F(\cdot,\omega)\in C^1(X,X_0)$, $F(v,\cdot)\in C^0(\Omega,X_0)$ ($v\in X$), $F^\omega(t,v)=F(v,\omega\cdot t)$ is H\"older continuous in $t$, and $F(v,\omega)=o(\|v\|)$ ($X_0=L^2(S^1)$), here  $X$ is as in the previous subsection. It is well-known that the solution operator $\Phi_t(\cdot,\omega)$ of \eqref{lipschitz} exists in the usual sense, that is, for any $v\in X$, $\Phi_0(v,\omega)=v$, $\Phi_t(v,\omega)\in \mathcal{D}(A(\omega\cdot t))$; $\Phi_t(v,\omega)$ is differentiable in $t$ with respect to $X_0$ norm and satisfies \eqref{lipschitz} for $t>0$.\par

Suppose that $\sigma(\Omega)=\cup_{k=0}^{\infty}I_k$ is the spectrum of \eqref{linear-opera}. The following lemma can be proved by using arguments as in \cite{P.Bates,Chow1991, Chow1994-2, Hen}.
\begin{lemma}\label{invari-mani}
 There is a $\delta_0>0$ such that for any $0<\delta^*<\delta_0$ and $0\leq n_1\leq n_2\leq\infty$, ($n_1\not = n_2$ when $n_2=\infty$), \eqref{lipschitz} admits for each $\omega\in \Omega$ a local invariant manifold $W^{n_1,n_2}(\omega,\delta^*)$ with the following properties.
 \begin{itemize}
 \item [\rm{(i)}] There are $K_0>0$, and a bounded continuous function $h^{n_1,n_2}(\omega): V^{n_1,n_2}(\omega)$ $\rightarrow V^{n_2+1,\infty}(\omega)\oplus V^{0,n_1-1}(\omega)) $ being $C^1$ for each fixed $\omega\in \Omega$, and $h^{n_1,n_2}(v,\omega)$ $=o(\|v\|)$, $\|(\partial h^{n_1,n_2}/\partial v)(v,\omega)\|\leq K_0$ for all $\omega\in \Omega$, $v\in V^{n_1,n_2}(\omega)$ such that
 \begin{eqnarray*}
  W^{n_1,n_2}(\omega,\delta^*)=\left\{ v_0^{n_1,n_2}+h^{n_1,n_2}(v_0^{n_1,n_2},\omega):v_0^{n_1,n_2}\in V^{n_1,n_2}(\omega) \cap \{v\in X: \|v\|<\delta^*\}\right\}.
 \end{eqnarray*}
 Moreover, $W^{n_1,n_2}(\omega,\delta^*)$ are diffeomorphic to $V^{n_1,n_2}(\omega)\cap\{v\in X| \|v\|<\delta^*\}$, and $W^{n_1,n_2}(\omega,\delta^*)$ are tangent to $V^{n_1,n_2}(\omega)$ at $0\in X$ for each $\omega\in \Omega$.
 \item [\rm{(ii)}] $W^{n_1,n_2}(\omega,\delta^*)$ is locally invariant in the sense that if $v\in W^{n_1,n_2}(\omega,\delta^*)$ and $\norm{\Phi_t(v,\omega)}<\delta^*$ for all $t\in [0,T]$, then $\Phi_t(v,\omega)\in W^{n_1,n_2}(\omega\cdot t,\delta^*)$ for all $t\in [0,T]$. Therefore,
 for any $v\in W^{n_1,n_2}(\omega,\delta^*)$, there is a $\tau>0$ such that $\Phi_t(v,\omega)\in W^{n_1,n_2}(\omega\cdot t,\delta^*)$ for any $t\in \mathbb{R}$ with $0<t<\tau$.
 \end{itemize}
\end{lemma}
 If $0\in I_{n_0}=[a_{n_0},b_{n_0}]\subset\sigma(\Omega)$, then $W^s(\omega,\delta^*)=W^{n_0+1,\infty}(\omega,\delta^*)$, $W^{cs}(\omega,\delta^*)=W^{n_0,\infty}(\omega,\delta^*)$, $W^c(\omega,\delta^*)=W^{n_0,n_0}(\omega,\delta^*)$, $W^{cu}(\omega,\delta^*)=W^{0,n_0}(\omega,\delta^*)$, and $W^u(\omega,\delta^*)=W^{0,n_0-1}(\omega,\delta^*)$ are referred to as {\it local stable, center stable, center, center unstable, and unstable manifolds} of \eqref{lipschitz} at $\omega\in \Omega$, respectively.\par

In the following remark, we will list some useful properties of local invariant manifolds.  These properties also can be proved by following the similar argument in \cite{P.Bates,Chow1991, Chow1994-2, Hen}.
\begin{remark}\label{overflow-invariant}
{\rm
(1) $W^s(\omega,\delta^*)$ and $W^u(\omega,\delta^*)$ are overflowing invariant in the sense that if $\delta^*$ is sufficiently small, then
\[
\Phi_t(W^s(\omega,\delta^*),\omega)\subset W^s(\omega\cdot t,\delta^*),
\]
for $t$ sufficiently positive, and
\[
\Phi_t(W^u(\omega,\delta^*),\omega)\subset W^u(\omega\cdot t,\delta^*),
\]
for $t$ sufficiently negative. $W^s(\omega,\delta^*)$ and $W^u(\omega,\delta^*)$  are unique and have the following characterizations:
there are $\delta_1^*,\delta_2^*>0$ such that
\begin{align*}
&\{v\in X\, :\, \|\Phi_t(v,\omega)\|\le \delta_1^*\,\,{\rm for}\,\, t\ge 0\,\textnormal{ and }\Phi_t(v,\omega)\to 0\,\, \textnormal{exponentially as}\,\, t\to\infty\}\\
&\subset
W^s(\omega,\delta^*)\subset \{v\in X\, :\, \|v\|\le \delta_2^*,\,\, \|\Phi_t(v,\omega)\|\to 0\,\, {\rm as}\,\, t\to\infty\}
\end{align*}
and
\begin{align*}
&\{v\in X\, :\,\textnormal{the backward orbit } \Phi_t(v,\omega) \textnormal{ exists and } \|\Phi_t(v,\omega)\|\le \delta_1^*\,\,{\rm for}\,\, t\le 0,\\
&\textnormal{ further, }\Phi_t(v,\omega)\to 0\,\, \textnormal{exponentially as}\,\, t\to -\infty\}\\
&\subset
W^u(\omega,\delta^*)\subset \{v\in X\, :\, \|v\|\le \delta_2^*,\,\, \|\Phi_t(v,\omega)\|\to 0\,\, {\rm as}\,\, t\to -\infty\}.
\end{align*}

Moreover, one can find constants $\alpha$, $C>0$, such that for any $\omega\in \Omega$, $v^s\in W^s(\omega,\delta^*)$, $v^u\in W^u(\omega,\delta^*)$,
\begin{equation}\label{exponen-decrea}
\begin{split}
  \|\Phi_t(v^s,\omega)\|&\leq Ce^{-\frac{\alpha}{2}t}\|v^s\|\quad \text{for}\ t\geq 0,\\
  \|\Phi_t(v^u,\omega)\|&\leq Ce^{\frac{\alpha}{2}t}\|v^u\|\quad \text{for}\ t\leq 0.
\end{split}
\end{equation}
\vbox{}

(2) $W^{cs}(\omega,\delta^*)$ (choose $\delta^*$ smaller if necessary) has a repulsion property in the sense that if $\norm{v}<\delta^*$ but $v\notin W^{cs}(\omega,\delta^*)$, then there is $T>0$ such that $\norm{\Phi_T(v,\omega)}\ge \delta^*$. Consequently, if $\norm{\Phi_t(v,\omega)}<\delta^*$ for all $t\ge 0$ then one may conclude that $v\in W^{cs}(\omega,\delta^*)$. Note that $W^{cs}(\omega,\delta^*)$ is not unique in general.

\vskip 3mm

(3) $W^{cu}(\omega,\delta^*)$ has an attracting property  in the sense that if $\|\Phi_t(v,\omega)\|<\delta^*$ for all $t\ge 0$, then $v^*\in W^{cu}(\omega^*,\delta^*)$ whenever $(\Phi_{t_n}(v,\omega),\omega\cdot t_n)\to (v^*,\omega^*)$ with some $t_n\to \infty$. Moreover, one can choose $\delta^*$ smaller such that, if  $\norm{v}<\delta^*$ with a unique backward orbit $\Phi_{t}(v,\omega)(t\le 0)$ but $v\notin W^{cu}(\omega,\delta^*)$, then there is $T<0$ such that $\norm{\Phi_{T}(v,\omega)}\ge \delta^*$. As a consequence, if $v$ has a unique backward orbit $\Phi_{t}(v,\omega)(t\le 0)$ with $\norm{\Phi_t(v,\omega)}<\delta^*$ for all $t\le 0$ then, one may conclude that $v\in W^{cu}(\omega,\delta^*)$. Note also that $W^{cu}(\omega,\delta^*)$ is not unique in general.

\vskip 3mm

(4) By the invariant foliation theory in \cite{Chow1991, Chow1994-2}, one has that for any $\omega\in \Omega$,
\[
 W^{cs}(\omega,\delta^*)={\cup}_{u_c\in W^c(\omega,\delta^*)}\bar{W}_s(u_c,\omega,\delta^*)\ ({\rm resp.}\, W^{cu}(\omega,\delta^*)={\cup}_{u_c\in W^c(\omega,\delta^*)}\bar{W}_u(u_c,\omega,\delta^*)),
\]
where $\bar{W}_s(u_c,\omega,\delta^*)$ (resp. $\bar{W}_u(u_c,\omega,\delta^*)$) is the so-called {\it stable leaf} (resp. {\it unstable leaf}) of \eqref{lipschitz} at $u_c$. It is invariant in the sense that if $\tau>0$ (resp. $\tau<0$) is such that $\Phi_t(u_c,\omega)\in W^{c}(\omega\cdot t,\delta^*)$ and $\Phi_t(u,\omega)\in W^{cs}(\omega,\delta^*)$ (resp. $\Phi_t(u,\omega)\in W^{cu}(\omega,\delta^*)$) for all $0\leq t<\tau$ (resp. $\tau<t\leq 0$), where $u\in \bar{W}_s(u_c,\omega,\delta^*)$ (resp. $u\in \bar{W}_u(u_c,\omega,\delta^*)$), then $\Phi_t(u,\omega)\in\bar{W}_s(\Phi_t(u_c,\omega),\omega\cdot t,\delta^*)$ (resp. $\Phi_t(u,\omega)\in\bar{W}_u(\Phi_t(u_c,\omega),\omega\cdot t,\delta^*)$) for $0\leq t<\tau$ (resp. $\tau<t\leq 0$). Moreover, there are $K,\beta >0$ such that for any $u\in\bar{W}_s(u_c,\omega,\delta^*)$ (resp. $u\in\bar{W}_u(u_c,\omega,\delta^*)$) and $\tau>0$ (resp. $\tau<0$) with $\Phi_t(u,\omega)\in W^{cs}(\omega\cdot t,\delta^*)$ (resp. $\Phi_t(u,\omega)\in W^{cu}(\omega\cdot t,\delta^*)$), $\Phi_t(u_c,\omega)\in W^{c}(\omega\cdot t,\delta^*)$ for $0\leq t<\tau$ (resp. $\tau<t\leq 0$), one has that
\begin{equation*}
\begin{split}
\|\Phi_t(u,\omega)-\Phi_t(u_c,\omega)\|&\le  Ke^{-\beta t}\|u-u_c\|\\
(\mathrm{resp}.\ \|\Phi_t(u,\omega)-\Phi_t(u_c,\omega)\|&\le  Ke^{\beta t}\|u-u_c\|)
\end{split}
\end{equation*}
for $0\leq t<\tau$ (resp. $\tau< t\leq 0$).
}

\end{remark}

\section{Structure of $\omega$-limit sets}

This section is devoted to the structure of the general $\omega$-limit sets of \eqref{equation-lim2}. Let $(u_0,g_0)\in X\times H(f)$ be such that the motion $\Pi^{t}(u_0,g_0)$($t\geq 0$) is bounded, we will prove the following theorem.

\begin{theorem}\label{structure}
Assume that $f(t,u,u_x)=f(t,u,-u_x)$. Then the $\omega$-limit set $\omega(u_0,g_0)$ of \eqref{equation-lim2} contains at most two minimal sets. Moreover, one of the following is true:
\begin{itemize}
\item[{ \rm (i)}] $\omega(u_0,g_0)$ is a minimal invariant set.
\item[{ \rm (ii)}] $\omega(u_0,g_0)=M_1\cup M_{11}$, where $M_1$ is minimal, $M_{11}\neq \emptyset$, and $M_{11}$ connects $M_1$ in the sense that if $(u_{11},g)\in M_{11}$, then $\omega(u_{11},g)\subset M_1$ and $\alpha (u_{11},g)\subset M_1$.

\item[{ \rm (iii)}]  $\omega(u_0,g_0)=M_1\cup M_2\cup M_{12}$, where $M_1$, $M_2$ are minimal sets, $M_{12}\not =\emptyset$, and for any $u_{12}\in M_{12}$, either  $M_1\subset \omega(u_{12},g)$ and $M_2\cap \omega(u_{12},g)=\emptyset$, or $M_2\subset \omega(u_{12},g)$ and $M_1\cap \omega(u_{12},g)=\emptyset$, or  $M_1\cup M_2\subset \omega(u_{12},g)$ (and analogous for $\alpha(u_{12},g)$).
\end{itemize}

\end{theorem}

\begin{remark}{\rm
 We remark that all the three cases in Theorem \ref{structure} may occur (see Examples in Section 6).}
\end{remark}

Before we proceed to prove Theorem \ref{structure}, we first give an important property of $\omega(u_0,g_0)$, which is motivated by the proof of \cite[Theorem B]{Chen1989160}.

\begin{proposition}\label{order}
Assume that $f(t,u,u_x)=f(t,u,-u_x)$ is $C^1$-admissible. Let $(u_0,g_0)\in X\times H(f)$ be such that the motion $\Pi^{t}(u_0,g_0)$($t>0$) is bounded and $\omega(u_0,g_0)$ be the $\omega$-limit set. Then, there is a point $x_0\in S^1$ such that for any $(u,g)\in \omega(u_0,g_0)$, one has $u_x(x_0)=0$.
\end{proposition}

In order to prove Proposition \ref{order}, we introduce the following notation: for any $u\in X$ and $a\in S^1$, we define function $\rho_a u$ as below:
\begin{equation}\label{reflection}
\rho_a u(\cdot)=u(2a-\cdot)
\end{equation}
For any $a\in S^1$, if $\varphi(t,\cdot;u,g)$ is a bounded classical solution of \eqref{equation-lim1}, then since $g(t,u,u_x)=g(t,u,-u_x)$, $\rho_a\varphi(t,\cdot;u,g)$ is also a bounded classical solution of \eqref{equation-lim1} and $\rho_a\varphi(t,\cdot;u,g)=\varphi(t,\cdot;\rho_au,g)$.

\begin{lemma}\label{symb}
  Assume that $f(t,u,u_x)=f(t,u,-u_x)$ is $C^1$-admissible. Let $\varphi(t,\cdot;u,g)$ be a bounded solution of \eqref{equation-lim1} with $\varphi(0,\cdot;u_0,g)=u_0\in X$. Then we have
  \begin{equation}\label{symb-keep}
  \lim_{t\to\infty}{\mathrm{sgn}(\varphi_x(t,a;u_0,g))}
  \end{equation}
  exists for every $a\in S^1$.
\end{lemma}
\begin{proof}
 Denote $u(t,x)=\varphi(t,x;u_0,g)$ and let $w(t,x)=\rho_a u(t,x)-u(t,x)$, then $w(t,x)$ satisfies
\begin{equation}\label{differen-equa}
  w_t=w_{xx}+b(t,x)w_x+c(t,x)w,
\end{equation}
  where
\begin{equation*}
 b(t,x)=-\int_0^1\dfrac{\partial g}{\partial p}(t,x,u(t,x),su_x(t,x)+(1-s)u_x(t,2a-x))ds
\end{equation*}
and
\begin{equation*}
  c(t,x)=-\int_0^1\dfrac{\partial g}{\partial u}(t,x,su(t,x)+(1-s)u(t,2a-x),{u}_x(t,2a-x))ds.
\end{equation*}
Moreover, for all $t\geq 0$, one has
\begin{equation*}
w(t,a)=0\ \text{and}\ w_x(t,a)=-2u_x(t,a).
\end{equation*}
If $w\equiv 0$, then $u_x(t,a)\equiv 0$ for $t\geq 0$ and \eqref{symb-keep} is established. Thus, we only consider the case where $w\not\equiv 0$. Since $g$ is $C^1$-admissible, equation \eqref{differen-equa} satisfies the assumption in Lemma 2.1. So, Lemma \ref{zero-number} (c), there exists $T>0$ such that, for all $t\geq T$, $w(t,\cdot)$ has only simple zeros on $S^1$. Since $w(t,a)=0$ for all $t\geq 0$, $w_x(t,a)=-2u_x(t,a)\neq 0$ for all $t\geq T$, which implies that $u_x(t,a)$ is always positive or always negative for $t\geq T$. Thus, $\mathrm{sign}(u_x(t,a))=1$ or $-1$ for all $t\geq T$). The proof of \eqref{symb-keep} is completed.
\end{proof}

A point $u\in X$ is called {\it spatially-homogeneous} if $u(\cdot)$ is independent of the spatial variable $x$. Otherwise, $u$ is called {\it spatially-inhomogeneous}.
\begin{lemma}\label{refle-prop}
Let all the hypotheses in Lemma \ref{symb} hold. Then for any point $(u,g)\in \omega(u_0,g_0)$ and $a\in S^1$, we have:
  \begin{itemize}
    \item[{ \rm (i)}] either $\rho_a u=u$ or $\rho_a u-u$ has only simple zeros on $S^1$;
    \item[{ \rm (ii)}] $\rho_a u=u$ if and only if $u_x(a)=0$.
  \end{itemize}
\end{lemma}
\begin{proof}
(i) Since $(u,g)\in \omega(u_0,g_0)$, there exists $t_n\to\infty$ such that $\varphi(t_n,\cdot;u_0,g_0)\to u$; and hence $\varphi(t_n,\cdot;\rho_au_0,g_0)\to \rho_au$. By virtue of Lemma \ref{sequence-limit}, $\rho_a u=u$; otherwise, $\rho_a u-u$ has only simple zeros on $S^1$. (ii) follows immediately from (i).
\end{proof}
\begin{proof}[Proof of Proposition \ref{order}]
 Given any $(u,g)\in \omega(u_0,g_0)$ with $u$ not being spatially-homogeneous. Choose $x_0\in S^1$ be a minimum point of $u$:
  \[
  u(x_0)=\min_{x\in S^1}u(x).
  \]
  Since $u$ is spatially-inhomogeneous, by virtue of Lemma \ref{refle-prop} (ii), one has $\{x\in S^1|u_x(x)=0\}$ is a discrete set (Otherwise, suppose that $a$ is an accumulation point of those $a_n$ for which $u_x(a_n)=u_x(a)=0$. Then, by Lemma \ref{refle-prop}(ii), $u(2a_n-x)=u(2a-x)=u(x)$ for each $x\in S^1$. So, putting $b_n=2a-2a_n$, we get a sequence $b_n\neq 0$, such that $b_n\to 0$ and $u(x)=u(x+b_n)$ for all $x\in S^1$. Consequently, $u$ is constant, a contradiction). It is therefore possible to find $x_1,x_2\in S^1$ such that $u_x(x)<0$ for $x\in (x_1,x_0)$ and $u_x(x)>0$ for $x\in (x_0,x_2)$. By Lemma \ref{symb}, for any other $(\tilde u,\tilde g)\in\omega(u_0,g_0)$, one has $\tilde u_x(x)\leq 0$ for $x\in (x_1,x_0)$ and $\tilde u_x(x)\geq 0$ for $x\in (x_0,x_2)$. Thus, $\tilde u_x(x_0)=0$. The proof is completed.
\end{proof}

Throughout this section, we always denote by $x_0$ the point in Proposition \ref{order} such that $u_x(x_0)=0$ for all $(u,g)\in \omega(u_0,g_0)$.

Before proving Theorem \ref{structure}, we still need the following two lemmas:
\begin{lemma}\label{constant-on-twosets}
Let $M_1,M_2\subset X\times H(f)$ be two minimal sets of \eqref{equation-lim2}. Then there is an integer $N>0$, such that for any $g\in H(f)$, $(u_i,g)\in M_i\cap p^{-1}(g)$($i=1,2$), one has
\begin{equation}\label{constant-on-twosets1}
  z(\varphi(t,\cdot;u_1,g)-\varphi(t,\cdot;u_2,g))=N,\quad \forall t\in \mathbb{R}.
\end{equation}
\end{lemma}
\begin{proof}
For any $g\in H(f)$, $(u_i,g)\in M_i\cap p^{-1}(g)$($i=1,2$), we {\it claim that there is an integer $N\in \mathbb {N}$ such that $z(\varphi(t,\cdot;u_1,g)-\varphi(t,\cdot;u_2,g))=N$ for all $t\in \mathbb {R}$}. To prove this claim, we first note that there are $T>0$ and $N_1, N_2$ such that
\begin{equation}\label{positive-constant}
 z(\varphi(t,\cdot;u_1,g)-\varphi(t,\cdot;u_2,g))=N_1 \quad \text{for all}\ t\ge T
\end{equation}
and
\begin{equation}\label{negative-constant}
 z(\varphi(t,\cdot;u_1,g)-\varphi(t,\cdot;u_2,g))=N_2 \quad \text{for all}\ t\le -T.
\end{equation}
In fact, \eqref{positive-constant} follows directly from Lemma \ref{difference-lapnumber}. As for \eqref{negative-constant}, one can take a sequence $t_n\to -\infty$ such that $\Pi^{t_n}(u_i,g)$($i=1,2$) converges to $(\tilde u_i,\tilde g)\in M_i\cap p^{-1}(\tilde g)$ as $n\to \infty$, for $i=1,2$. By Lemma \ref{sequence-limit},  there is an $N_2\in\mathbb{N}$ such that
\begin{equation}\label{limit-const}
  z(\varphi(t,\cdot;\tilde u_1,\tilde g)-\varphi(t,\cdot;\tilde u_2,\tilde g))=N_2,\quad\text{for all}\ t\in \mathbb{R}.
\end{equation}
Therefore,
\begin{equation}\label{asympto}
 z(\varphi(t_n,\cdot;u_1,g)-\varphi(t_n,\cdot;u_2,g))=N_2
\end{equation}
for $n$ sufficiently large; and hence, Lemma \ref{difference-lapnumber}(a) implies that \eqref{negative-constant} holds.

We now turn to prove that $N_1=N_2$. Choose a sequence $t_n\to \infty$ such that $\Pi^{t_n}(u_2,g)\to (u_2,g)$ as $n\to \infty$. Without loss of generality, we assume that $\Pi^{t_n}(u_1,g)\to (\bar u_1,g)$. By Lemma \ref{sequence-limit}, there is an integer $N>0$ satisfying that
\begin{equation}\label{constant-1}
  z(\varphi(t,\cdot;\bar u_1,g)-\varphi(t,\cdot;u_2,g))=N, \quad \forall t\in\mathbb{R}.
\end{equation}
Since $(u_1,g)$, $(\bar u_1,g)\in M_1\cap p^{-1}(g)$, it then follows from \cite[Theorem 4.2]{SWZ} that there are $(u^*_1,g^*)\in M_1\cap p^{-1}(g)$ and sequence $t^*_n\to\infty$ ($s^*_n\to \infty$) as $n\to \infty$ such that $\Pi^{t^*_n}(u_1,g)\to (u^*_1,g^*)$ and $\Pi^{t^*_n}(\bar u_1,g)\to (u^*_1,g^*)$($\Pi^{-s^*_n}(u_1,g)\to (u^*_1,g^*)$ and $\Pi^{-s^*_n}(\bar u_1,g)\to (u^*_1,g^*)$) as $n\to \infty$. By \eqref{negative-constant}-\eqref{constant-1} and the continuity of $z(\cdot)$, we get $N_1=N=N_2$. That is,
\begin{equation}\label{constant-2}
  z(\varphi(t,\cdot;u_1,g)-\varphi(t,\cdot;u_2,g))=N,\quad \forall\ |t|\ge T.
\end{equation}
By Lemma \ref{difference-lapnumber}(a), we have proved our claim.

Finally, we show that $N$ is actually independent of $g\in H(f)$ and $(u_i,g)\in M_i\cap p^{-1}(g)$ ($i=1,2$). Indeed, for any $g\in H(f)$ and any $(u_i,g)$, $(\hat u_i,g)\in M_i\cap p^{-1}(g)$($i=1,2$), By the claim above, there are $N_1,N_2\in\mathbb{N}$ such that
\begin{equation*}
  z(\varphi(t,\cdot;u_1,g)-\varphi(t,\cdot;u_2,g))=N_1,\quad \text{for all}\ t\in \mathbb{R},
\end{equation*}
and
\begin{equation*}
  z(\varphi(t,\cdot;\hat u_1,g)-\varphi(t,\cdot;\hat u_2,g))=N_2,\quad \text{for all}\ t\in \mathbb{R}.
\end{equation*}
It again follows from \cite[Theorem 4.2]{SWZ} that $(u_i,g)$, $(\hat u_i,g)$ forms a two sided proximal pair. Thus, by the continuity of $z(\cdot)$, one has
\begin{eqnarray*}
\begin{split}
  N_1=z(\varphi(t,\cdot; u_1,g)-\varphi(t,\cdot;u_2,g))&=z(\varphi(t,\cdot; u_1,g)-\varphi(t,\cdot;\hat u_2,g))\\ &=z(\varphi(t,\cdot;\hat u_1,g)-\varphi(t,\cdot;\hat u_2,g))=N_2.
\end{split}
\end{eqnarray*}
Moreover, for any $g,\hat g\in H(f)$ and $(u_i,g)\in M_i\cap p^{-1}(g)$, $(\hat u_i,\hat g)\in M_i\cap p^{-1}(\hat g)$($i=1,2$). Again, one can choose a sequence  $t_n\to -\infty$ and $(\bar u_ 2, \hat g)\in M_2\cap p^{-1}(\hat g)$ such that $\Pi^{t_n}(u_1,g)\to (\hat u_1,\hat g)$ and $\Pi^{t_n}(u_2,g)\to (\bar u_2,\hat g)$ as $n\to \infty$. Similarly as the arguments in \eqref{limit-const}-\eqref{asympto}, we have
\begin{eqnarray*}
\begin{split}
  N=z(\varphi(t,\cdot; u_1,g)-\varphi(t,\cdot;u_2,g))=z(\varphi(t,\cdot;\hat u_1,\hat g)-\varphi(t,\cdot;\hat u_2,\hat g)),
\end{split}
\end{eqnarray*}
for all $t\in \mathbb{R}$. Thus, we have proved that $N$ is independent of $g\in H(f)$ and $(u_i,g)\in M_i\cap p^{-1}(g)$($i=1,2$), which completes the proof of the lemma.
\end{proof}

\begin{lemma}\label{separated}
 Let $\omega(u_0,g_0)\subset X\times H(f)$ be in Theorem \ref{structure} and $M_1,M_2\subset \omega(u_0,g_0)$ be two minimal sets. Define
 \begin{eqnarray}
 \begin{split}
   m_i(g):=\min\{u(x_0)|(u,g)\in M_i\cap p^{-1}(g)\},\\
   M_i(g):=\max\{u(x_0)|(u,g)\in M_i\cap p^{-1}(g)\}
 \end{split}
 \end{eqnarray}
 for $i=1,2$. Then $M_1$, $M_2$ are separated in the following sense:
 \begin{itemize}
   \item[{ \rm (i)}] $[m_1(g),M_1(g)]\cap[m_2(g),M_2(g)]=\emptyset$ for all $g\in H(f)$;
   \item[{ \rm (ii)}] If $m_2(\tilde g)>M_1(\tilde g)$ for some $\tilde g\in H(f)$, then there exists $\delta>0$ such that  $m_2(g)>M_1(g)+\delta$ for all $g\in H(f)$.
 \end{itemize}
\end{lemma}
\begin{proof}
We first claim that (i) holds for some $\tilde g\in H(f)$. Otherwise, one has $m_2(g)\le M_1(g)$ and $m_1(g)\le M_2(g)$ for all $g\in H(f)$. Given $g,g^*\in H(f)$, let $(u_1,g)\in M_1\cap p^{-1}(g)$ be such that $u_1(x_0)=m_1(g)$ and $(u_2,g)\in M_2\cap p^{-1}(g)$ be such that $u_2(x_0)=M_2(g)$. It follows from Lemma \ref{constant-on-twosets} that $z(\varphi(t,\cdot;u_1,g)-\varphi(t,\cdot;u_2,g))\equiv constant$, for all $t\in \mathbb{R}$. By Lemma \ref{difference-lapnumber}(b), we have $\varphi(t,\cdot;u_1,g)-\varphi(t,\cdot;u_2,g)(x_0)\neq 0$ for all $t\in \mathbb{R}$. Note that $u_1(x_0)=m_1(g)\le M_2(g)=u_2(x_0)$. Then, one has
\begin{equation}\label{smaller}
  \varphi(t,\cdot;u_1,g)(x_0)<\varphi(t,\cdot;u_2,g)(x_0),\quad \text{for all}\ t\in\mathbb{R}.
\end{equation}
By the minimality of $M_1$, there is a sequence $t_n\to \infty$, such that $\Pi^{t_n}(u_1,g)\to (u^*_1,g^*)$ as $n\to \infty$, where $(u^*_1,g^*)\in M_1$ with $u^*_1(x_0)=M_1(g^*)$. Without loss of generality, we can also assume that $\Pi^{t_n}(u_2,g)\to (u^*_2,g^*)$ as $n\to \infty$. By \eqref{smaller}, $M_1(g^*)=u^*_1(x_0)\leq u^*_2(x_0)\le M_2(g^*)$. Moreover, it again follows from Lemma \ref{constant-on-twosets} that $M_1(g^*)\neq M_2(g^*)$, which implies that $M_1(g^*)<M_2(g^*)$ (Otherwise, $x_0$ is a zero of $u_1^*-u_2^*$  with multiplicity two) so by Lemma \ref{difference-lapnumber}(b) $z(\varphi(t,\cdot;u^*_1,g^*)-\varphi(t,\cdot;u^*_1,g^*))$ drops at $t=0$, a contradiction). However, similarly as above, one can also obtain that $M_2(g^*)<M_1(g^*)$, a contradiction.

Without loss of generality, we assume that $M_1(\tilde g)<m_2(\tilde g)$ for some $\tilde g\in H(f)$. Then we will show that $M_1(g)<m_2(g)$ for all $g\in H(f)$. Suppose that there is a $g^*\in H(f)$ such that $M_1(g^*)\ge m_2(g^*)$. Let $(u^*_2,g^*)\in M_2\cap p^{-1}(g^*)$ with $u^*_2(x_0)=m_2(g^*)$ and choose $(u^*_1,g^*)\in M_1\cap p^{-1}(g^*)$ with $u^*_1(x_0)=M_1(g^*)$. By the minimality of $M_2$, we can find a sequence $\{t_n\}$, $t_n\to \infty$, such that $\Pi^{t_n}(u^*_2,g^*)\to (u^{**}_2,\tilde g)$ as $n\to \infty$ and $u^{**}_2(x_0)=m_2(\tilde g)$. Without loss of generality, one may also assume that $\Pi^{t_n}(u^*_1,g^*)\to (u^{**}_1,\tilde g)$ as $n\to \infty$. By the same arguments as in the previous paragraph, one has
\begin{equation*}
  m_2(\tilde g)=u^{**}_2(x_0)\le u^{**}_1(x_0)\le M_1(\tilde g),
\end{equation*}
contradicting our assumption. Therefore, $m_2(g)>M_1(g)$ for all $g\in H(f)$. We have proved (i).

(ii) By (i), it is clear that $m_2(g)>M_1(g)$ for all $g\in H(f)$. Suppose that there is sequence  $m_2(g_n)>M_1(g_n)$ satisfying that $m_2(g_n)-M_1(g_n)\to 0$ as $n\to \infty$. Then, we may assume that $g_n\to g^*\in H(f)$, $m_2(g_n)\to c$ and $M_1(g_n)\to c$ as $n\to\infty$, for some $c\in\mathbb{R}$. Since $M_i$ ($i=1,2$) are compact, $c\in[m_1(g^*),M_1(g^*)]\cap[m_2(g^*),M_2(g^*)]$, a contradiction.
\end{proof}

\begin{proof}[Proof of Theorem \ref{structure}]
  Suppose that $\omega(u_0,g_0)$ contains three minimal sets $M_i$ ($i=1,2,3$). Define
\begin{equation*}
  A_i(g)=\{u(x_0)|(u,g)\in M_i\cap p^{-1}(g)\}
\end{equation*}
and
\begin{equation*}
  m_i(g)=\min A_i(g),\quad M_i(g)=\max A_i(g),
\end{equation*}
for all $g\in H(f)$ and $i=1,2,3$. By virtue of Lemma \ref{separated}(ii), we may assume without loss of generality that there is a $\delta>0$ such that
\begin{equation}\label{inequality}
  M_1(g)+\delta\leq m_2(g)\le M_2(g)<M_2(g)+\delta\le m_3(g),
\end{equation}
for all $g\in H(f)$.

Now choose $(u_i,g_0)\in M_i\cap p^{-1}(g_0)$, $i=1,2,3$, and consider $(u_0,g_0)$ and $(u_2,g_0)$. It follows from Lemma \ref{difference-lapnumber} that there is a $T>0$ such that $z(\varphi(t,\cdot;u_0,g_0)-\varphi(t,\cdot;u_2,g_0))=constant$, for all $t\ge T$. By Lemma \ref{difference-lapnumber}(b), we can assume without loss of generality that $\varphi(t,\cdot;u_0,g_0)(x_0)<\varphi(t,\cdot;u_2,g_0)(x_0)$ for all $t\ge T$. Since $M_3\subset \omega(u_0,g_0)$, there is a sequence $\{t_n\}$, $t_n\to \infty$, such that $\varphi(t_n,\cdot;u_0,g_0)(x_0)\to m_3(g^*)$ as $n\to \infty$. Let $\varphi(t_n,\cdot;u_2,g_0)(x_0)\to \beta(g^*)$ with $\beta(g^*)\in [m_2(g^*),M_2(g^*)]$. Consequently,
\[
m_3(g^*)\leq \beta(g^*)\le M_2(g^*),
\]
contradicting \eqref{inequality}. Thus, $\omega(u_0,g_0)$ contains at most two minimal sets. Let $\omega(u_0,g_0)=M_1\cup M_2\cup M_{12}$, where $M_1$, $M_2$ are minimal sets. If $M_1\neq M_2$, since $\omega(u_0,g_0)$ is connected, $M_{12}\neq \emptyset$. Choose $(u_{12},g)\in M_{12}$, it is clear that $\omega(u_{12},g)\cap (M_1\cup M_2)$ and $\alpha(u_{12},g)\cap (M_1\cup M_2)$ are not empty, for otherwise, either $\omega(u_{12},g)$ or $\alpha(u_{12},g)$ would contain a minimal set and therefore $\omega(u_0,g_0)$ would have three minimal sets. For the case $\omega(u_0,g_0)$ contains only one minimal set, that is, $M_1=M_2$. If $M_{12}\neq \emptyset$, then a similar argument shows that $\omega(u_{12},g)\cap M_1\neq \emptyset$, $\alpha(u_{12},g)\cap M_1\neq \emptyset$ for any $(u_{12},g)\in M_{12}$. Thus, we have completed our proof.
\end{proof}

\section{Hyperbolic $\omega$-limit sets}

By virtue of Theorem \ref{structure}, we will focus on structure of the hyperbolic $\omega$-limit sets of equation \eqref{equation-lim2}. To state the main result in this section, we need to further introduce some additional notations on the invariant manifolds associated with \eqref{equation-lim1}-\eqref{equation-lim2}.

Let $E\subset X\times H(f)$ be a connected and compact invariant set of \eqref{equation-lim2}, and $\varphi(t,\cdot;u_0,g)$ be the solution of \eqref{equation-lim1} with $\varphi(0,\cdot,u_0,g)=u_0(\cdot)$.
For any $\omega=(u_0,g)\in E$, we write $\omega\cdot t=\Pi^t(u_0,g)$. Consider the transformation $v=u-\varphi(t,\cdot;u_0,g)$ in \eqref{equation-lim1}. It turns out that the new variable $v$ satisfies the following equation:
\begin{equation}\label{variation-equation}
v_t=v_{xx}+a(x,\omega\cdot t)v_x+b(x,\omega\cdot t)v+F(v,\omega\cdot t),\,\,t>0,\,x\in S^{1}=\mathbb{R}/2\pi \mathbb{Z},
\end{equation}
where
$$F(v,\omega)=g(0,v+u_0,v_x+(u_0)_x)-g(0,u_0,(u_0)_x)-a(x,\omega)v_x-b(x,\omega)v,
$$
 $a(x,\omega)=g_p(0,u_0,(u_0)_x)$ (here $g_p(\cdot,\cdot,p)$ is the derivative of $g$ with respect to $p$), $b(x,\omega)=g_u(0,u_0,(u_0)_x)$.

Denote $A(\omega)=\frac{\partial^2}{\partial x^2}+a(\cdot,\omega)\frac{\p }{\p x}+b(\cdot,\omega)$. Then \eqref{variation-equation} can be rewritten as
\begin{equation}\label{linear-opera2}
v'=A(\omega\cdot t)v+F(v,\omega\cdot t).
\end{equation}
 Let $\sigma(E)=\cup_{k=0}^{\infty}I_k$ ($I_k$ is ordered from right to left) be the Sacker-Sell spectrum of the linear equation associated with \eqref{linear-opera2}:
\begin{equation}\label{associate-eq}
v'=A(\omega\cdot t)v,\,\,t>0,\,\omega\in E,\,v\in X.
\end{equation}
For any given $0\leq n_1\leq n_2\leq \infty$ ($n_1\not =n_2$ when $n_2=\infty$), let $V^{n_1,n_2}(\omega)$ be the invariant subspace of \eqref{associate-eq} associated with the spectrum set $\cup_{k=n_1}^{n_2}I_k$ at $\omega\in E$. Moreover, by Lemma \ref{invari-mani}, there is a well-defined local invariant manifold $W^{n_1,n_2}(\omega,\delta^*)$ of \eqref{linear-opera2}. Let
\begin{equation}\label{gener-inva}
M^{n_1,n_2}(\omega,\delta^*)=\{u\in X|u-u_0\in W^{n_1,n_2}(\omega,\delta^*)\}.
\end{equation}
$M^{n_1,n_2}(\omega,\delta^*)$ is then referred as a {\it local invariant manifold} of \eqref{equation-lim2} at $(u_0,g)$.\par

Assume $0\in \sigma(E)$ and $n_0$ is such that $I_{n_0}=[a_{n_0},b_{n_0}]\subset \sigma(E)$ with $a_{n_0}\leq 0\leq b_{n_0}$. Let  $V^s(\omega)$, $V^{cs}(\omega)$, $V^{c}(\omega)$, $V^{cu}(\omega)$, and $V^u(\omega)$  be the
 {\it stable, center stable, center, center unstable}, and {\it unstable subspaces} of \eqref{associate-eq} at $\omega\in E$, respectively. So, $W^s(\omega,\delta^*)$, $W^{cs}(\omega,\delta^*)$, $W^c(\omega,\delta^*)$, $W^{cu}(\omega,\delta^*)$ and $W^u(\omega,\delta^*)$ are well defined as in subsection \ref{invari-mani0}.  By virtue of \eqref{gener-inva}, we can define
$$M^{l}(\omega,\delta^*)=\{u\in X|u-u_0\in W^l(\omega,\delta^*)\}$$ for $l=s,cs,c,cu$ and $u$.
Then $M^s(\omega,\delta^*)$, $M^{cs}(\omega,\delta^*)$, $M^{c}(\omega,\delta^*)$, $M^{cu}(\omega,\delta^*)$ and $M^{u}(\omega,\delta^*)$ are continuous in $\omega\in E$ and referred as {\it local stable, center stable, center, center unstable}, and {\it unstable manifolds} of \eqref{equation-lim2} at $\omega=(u_0,g)\in E$, respectively.\par

Let $\omega=(u,g)\in E$, then the following Remark is directly from Remark \ref{overflow-invariant}.
\begin{remark}\label{stable-leaf}
{\rm
(1) $M^s(\omega,\delta^*)$ and $M^u(\omega,\delta^*)$ are overflowing invariant in the sense that if $\delta^*$ is sufficiently small, then
\[
\varphi(t,M^s(\omega,\delta^*),g)\subset M^s(\omega\cdot t,\delta^*),
\]
for $t$ sufficiently positive, and
\[
\varphi(t,M^u(\omega,\delta^*),g)\subset M^u(\omega\cdot t,\delta^*),
\]
for $t$ sufficiently negative.  $M^s(\omega,\delta^*)$ and $M^u(\omega,\delta^*)$  are unique and have the following characterizations:
there are $\delta_1^*,\delta_2^*>0$ such that
\begin{align*}
&\{v\in X\, :\, \|\varphi(t,\cdot;v,g)-\varphi(t,\cdot;u,g)\|\le \delta_1^*\,\,{\rm for}\,\, t\ge 0\,\textnormal{ and }\varphi(t,\cdot;v,g)-\varphi(t,\cdot;u,g)\to 0\,\\ &\textnormal{exponentially as}\,\, t\to\infty\}\\
&\subset
M^s(\omega,\delta^*)\subset \{v\in X\, :\, \|v-u\|\le \delta_2^*,\,\, \|\varphi(t,\cdot;v,g)-\varphi(t,\cdot;u,g)\|\to 0\,\, {\rm as}\,\, t\to\infty\}
\end{align*}
and
\begin{align*}
&\{v\in X\, :\,\textnormal{the backward orbit } \varphi(t,\cdot;v,g) \textnormal{ exists and } \|\varphi(t,\cdot;v,g)-\varphi(t,\cdot;u,g)\|\le \delta_1^*\,\,{\rm for}\,\, t\le 0,\\
&\textnormal{ further, }\varphi(t,\cdot;v,g)-\varphi(t,\cdot;u,g)\to 0\,\, \textnormal{exponentially as}\,\, t\to -\infty\}\\
&\subset
M^u(\omega,\delta^*)\subset \{v\in X\, :\, \|v-u\|\le \delta_2^*,\, \|\varphi(t,\cdot;v,g)-\varphi(t,\cdot;u,g)\|\to 0\,\, {\rm as}\,\, t\to -\infty\}.
\end{align*}

Moreover, one can find constants $\alpha$, $C>0$, such that for any $\omega\in E$, $v^s\in M^s(\omega,\delta^*)$, $v^u\in M^u(\omega,\delta^*)$,
\begin{equation}\label{exponen-decrea}
\begin{split}
  \|\varphi(t,\cdot;v^s,g)-\varphi(t,\cdot;u,g)\|&\leq Ce^{-\frac{\alpha}{2}t}\|v^s-u\|\quad \text{for}\ t\geq 0,\\
  \|\varphi(t,\cdot;v^u,g)-\varphi(t,\cdot;u,g)\|&\leq Ce^{\frac{\alpha}{2}t}\|v^u-u\|\quad \text{for}\ t\leq 0.
\end{split}
\end{equation}

\vskip 3mm
(2) $M^{cs}(\omega,\delta^*)$ (choose $\delta^*$ smaller if necessary) has a repulsion property in the sense that if $\norm{v-u}<\delta^*$ but $v\notin M^{cs}(\omega,\delta^*)$, then there is $T>0$ such that $\norm{\varphi(T,\cdot;v,g)-\varphi(T,\cdot;u,g)}\ge \delta^*$. Consequently, if $\norm{\varphi(t,\cdot;v,g)-\varphi(t,\cdot;u,g)}<\delta^*$ for all $t\ge 0$ then one may conclude that $v\in M^{cs}(\omega,\delta^*)$. Note that $M^{cs}(\omega,\delta^*)$ is not unique in general.

\vskip 3mm

(3) $M^{cu}(\omega,\delta^*)$ has an attracting property  in the sense that if $\|\varphi(t,\cdot;v,g)-\varphi(t,\cdot;u,g)\|<\delta^*$ for all $t\ge 0$, then $v^*\in M^{cu}(\omega^*,\delta^*)$ whenever $(\varphi(t_n,\cdot;v,g),\omega\cdot t_n)\to (v^*,\omega^*)$ and with some $t_n\to \infty$. Moreover, one can choose $\delta^*$ smaller such that, if  $\norm{v-u}<\delta^*$ with a unique backward orbit $\varphi(t,\cdot;v,g)(t\le 0)$ but $v\notin M^{cu}(\omega,\delta^*)$, then there is $T<0$ such that $\norm{\varphi(t,\cdot;v,g)-\varphi(t,\cdot;u,g)}\ge \delta^*$. As a consequence, if $v$ has a unique backward orbit $\varphi(t,\cdot;v,g)(t\le 0)$ with $\norm{\varphi(t,\cdot;v,g)-\varphi(t,\cdot;u,g)}<\delta^*$ for all $t\le 0$, then one may conclude that $v\in M^{cu}(\omega,\delta^*)$. Note that $M^{cu}(\omega,\delta^*)$ is not unique in general.

\vskip 3mm
(4) For any $\omega\in E$, we have
\[
 M^{cs}(\omega,\delta^*)={\cup}_{u_c\in M^c(\omega,\delta^*)}\bar{M}_s(u_c,\omega,\delta^*)\ ({\rm resp. }\, \ M^{cu}(\omega,\delta^*)={\cup}_{u_c\in M^c(\omega,\delta^*)}\bar{M}_u(u_c,\omega,\delta^*)),
\]
where $\bar{M}_s(u_c,\omega,\delta^*)$ (resp. $\bar{M}_u(u_c,\omega,\delta^*)$) is the so-called {\it stable leaf} (resp. {\it unstable leaf}) of \eqref{lipschitz} at $u_c$. It is invariant in the sense that if $\tau>0$ (resp. $\tau<0$) is such that $\varphi(t,\cdot;u_c,g)\in M^{c}(\omega\cdot t,\delta^*)$ and $\varphi(t,\cdot;v,g)\in M^{cs}(\omega,\delta^*)$ (resp. $\varphi(t,\cdot;v,g)\in M^{cu}(\omega,\delta^*)$) for all $0\leq t<\tau$ (resp. $\tau<t\leq 0$), where $v\in \bar{M}_s(u_c,\omega,\delta^*)$ (resp. $v\in \bar{M}_u(u_c,\omega,\delta^*)$), then $\varphi(t,\cdot;v,g)\in\bar{M}_s(\varphi(t,\cdot;u_c,g),\omega\cdot t,\delta^*)$ (resp. $\varphi(t,\cdot;v,g)\in\bar{M}_u(\varphi(t,\cdot;u_c,g),\omega\cdot t,\delta^*)$) for $0\leq t<\tau$ (resp. $\tau<t\leq 0$). Moreover, there are $K,\beta >0$ such that for any $u\in\bar{M}_s(u_c,\omega,\delta^*)$ (resp. $u\in\bar{M}_u(u_c,\omega,\delta^*)$) and $\tau>0$ (resp. $\tau<0$) with $\varphi(t,\cdot;v,g)\in M^{cs}(\omega\cdot t,\delta^*)$ (resp. $\varphi(t,\cdot;v,g)\in M^{cu}(\omega\cdot t,\delta^*)$), $\varphi(t,\cdot;u_c,g)\in M^{c}(\omega\cdot t,\delta^*)$ for $0\leq t<\tau$ (resp. $\tau<t\leq 0$), one has that
\begin{equation*}
\begin{split}
\|\varphi(t,\cdot;v,g)-\varphi(t,\cdot;u_c,g)\|&\le  Ke^{-\beta t}\|v-u_c\|\\
(\mathrm{resp}. \ \|\varphi(t,\cdot;v,g)-\varphi(t,\cdot;u_c,g)\|&\le  Ke^{\beta t}\|v-u_c\|)
\end{split}
\end{equation*}
for $0\leq t<\tau$ (resp. $\tau< t\leq 0$).

}
\end{remark}

\begin{remark}\label{invari-space}
{\rm
For any minimal set $M\subset E$, one has $\sigma(M)\subset \sigma(E)$ and $\dim V^u(M)\geq\dim V^u(E)$, $\dim V^c(M)\leq\dim V^c(E)$ and $\mathrm{codim}V^s(M)\leq \mathrm{codim}V^s(E)$ (here $V^u(M)$, $V^c(M)$ and $V^s(M)$ are {\it stable space, center space and unstable space} of the linearized variational equation  \eqref{associate-eq} on $M$).}
\end{remark}

 A subset $A\subset X$ is called {\it spatially-homogeneous} (resp. {\it spatially-inhomogeneous}) if any point in $A$ is spatially-homogeneous (resp. spatially-inhomogeneous). In particular, it deserves to point out that if $A$ is minimal, then $A$ is either spatially-inhomogeneous; or otherwise, $A$ is spatially-homogeneous.

Our main result in this section is the following

\begin{theorem}\label{hyperbolic0}
  Assume that $f(t,u,u_x)=f(t,u,-u_x)$.
   Let $\Omega=\omega(u_0,g_0)$ be an $\omega$-limit set of \eqref{equation-lim2} with $\dim V^c(\Omega)=0$. Then $\omega(u_0,g_0)$ is a spatially-homogeneous $1$-cover of $H(f)$.
\end{theorem}

 We first show that Theorem \ref{hyperbolic0} holds for any minimal sets.

\begin{lemma}\label{hyper-minimal}
 Assume $M\subset X\times H(f)$ be a minimal set of \eqref{equation-lim2} with $\dim V^c(M)=0$, then $M$ is spatially-homogeneous and $1$-cover of $H(f)$.
\end{lemma}
\begin{proof}
 This lemma can be found in \cite[Theorem 4.1.3,p.31]{SWZ}. For the sake of completeness, we give a detailed proof below.

Suppose that $M$ is spatially-inhomogeneous. Then for any $\omega=(u_0,g_0)\in M$, $\varphi_x(t,x;u_0,g_0)$ is a nontrivial solution of \eqref{associate-eq}, where $E$ is replaced by $M$. Recall that $f\in C^2$, then $\norm{\varphi_x(t,\cdot;u_0,g_0)}$ is bounded for all $t\in \mathbb{R}$. So, \eqref{twoside-estimate} implies that $(u_0)_x\in V^c(\omega)$, a contradiction to that $ \dim V^c(M)=0$. Thus, we have proved that $M$ is spatially-homogeneous.

Therefore, $M$ is also a minimal set of
\begin{equation}
\label{scalar-ode}
\dot u=\tilde g(t,u),
\end{equation}
where $\tilde g(t,u)=g(t,u,0)$ with $g\in H(f)$. It then follows from \cite[Theorem III.3.4]{Shen1998} that $M$ is an almost $1$-cover of $H(f)$. Moreover, when $\dim V^u(M)>0$, \eqref{scalar-ode} can be also viewed as a special case of \eqref{equation-lim1} under the Neumann boundary condition. Thus, by \cite[Corollary 4.6]{ShenYi-2}, $M$ is $1$-cover of $H(f)$. When $\dim V^u(M)=0$, $M$ is uniformly stable. Then \cite[Theorem II.2.8]{Shen1998} implies that $M$ is distal. Combining with $M$ is almost $1$-cover of $H(f)$, $M$ is $1$-cover of $H(f)$.
We have completed the proof of this lemma.
\end{proof}

 \begin{remark}\label{R:the-1}
{\rm
 (1) By virtue of Lemma \ref{hyper-minimal}, we immediately obtain that any minimal set $M\subset\O$ is spatially-homogeneous and a $1$-cover, provided that $\dim V^c(\O)=0$. However, it is not clear that whether the $\O$ itself is spatially-homogeneous. Our Theorem \ref{hyperbolic0} will confirm that the $\O$ itself is indeed a spatially-homogeneous $1$-cover of $H(f)$.

 (2) In fact, we can further obtain that
\begin{equation}\label{lyapunov-zero}
 \lim_{|t|\to\infty}\frac{\ln\|\varphi_x(t,x;u_0,g_0)\|}{t}=0
\end{equation}
for any $(u_0,g_0)\in M$, provided that $M$ is spatially-inhomogeneous. Recall that $\norm{\varphi_x(t,\cdot;u_0,g_0)}$ is bounded as in the proof of Lemma \ref{hyper-minimal}. Then \eqref{lyapunov-zero} is directly due to the fact that one can find a $\delta_0>0$ such that $\|\varphi_x(t,\cdot;u_0,g_0)\|\ge \delta_0$ for all $t\in\mathbb{R}$. (We can show the existence of such $\delta_0$ as follows: Without loss of generality, suppose that there is a sequence $t_n\to\infty$ such that $\|\varphi_x(t_n,\cdot;u_0,g_0)\|\to 0$ as $n\to\infty$, which entails that $\varphi_x(t_n,x;u_0,g)\to 0$ as $n\to\infty$ uniformly in $x\in S^1$. Without loss of generality, we assume that $(\varphi(t_n,\cdot;u_0,g),g\cdot{t_n})\to (u^*,g^*)\in M$
as $n\to\infty$. Thus, $u^*$ must be spatially-homogeneous, a contradiction.)}
\end{remark}

If we assume additionally that $\dim V^u(\Omega)=0$ in Theorem \ref{hyperbolic0}, then $\O$ is uniformly stable. Therefore, by virtue of  \cite[Theorem II.2.8]{Shen1998}, it is clear that $\omega(u_0,g_0)$ is spatially-homogeneous and a $1$-cover of $H(f)$.

Consequently, in the remaining of this section, we only assume that $\dim V^u(\Omega)>0$ in Theorem \ref{hyperbolic0}. We first need some technical lemmas. Let $E$ be the connected and compact invariant set as in \eqref{variation-equation}.

\begin{lemma}\label{zerocenter}
  Let $\omega=(u_0,g)\in E$ and \begin{equation}\label{L:lap-classify}
N_u=\left\{
\begin{split}
 &{\dim} V^u(E),\,\quad\,\,\,\text{ if }{\rm dim}V^u(E)\text{ is even,}\\
 &{\rm dim}V^u(E)+1,\,\text{ if }{\rm dim}V^u(E)\text{ is odd.}
\end{split}\right.
\end{equation}
Assume that  ${\rm dim}V^u(E)\ge 1$.
Then for $\delta^*>0$ small enough, one has
\begin{itemize}

\item[{\rm (1)}] If ${\rm dim}V^c(E)=0$ and ${\rm dim}V^u(E)$ is odd, then
\begin{eqnarray*}
\begin{split}
 & z(u(\cdot)-u_0(\cdot))\geq N_u\quad \text{for }u\in M^{s}(\omega,\delta^*)\setminus\{u_0\},\\
 & z(u(\cdot)-u_0(\cdot))\leq N_u-2\quad \text{for }u\in M^{u}(\omega,\delta^*)\setminus\{u_0\};\\
\end{split}
\end{eqnarray*}

\item[{\rm (2)}] If ${\rm dim}V^c(E)=1$ and $\dim V^u(E)$ is odd, then
\begin{eqnarray*}
\begin{split}
 %& z(u(\cdot)-u_0(\cdot))\geq N_u+3\quad \text{for }u\in M^s(\omega,\delta^*)\setminus\{u_0\}\\
 & z(u(\cdot)-u_0(\cdot))\geq N_u\quad \text{for }u\in M^{s}(\omega,\delta^*)\setminus\{u_0\},\\
 & z(u(\cdot)-u_0(\cdot))=N_u\quad \text{for }u\in M^{c}(\omega,\delta^*)\setminus\{u_0\},\\
 & z(u(\cdot)-u_0(\cdot))\leq N_u-2\quad \text{for }u\in M^{u}(\omega,\delta^*)\setminus\{u_0\};\\
 %& z(u(\cdot)-u_0(\cdot))\geq N_u-1\quad \text{for }u\in M^{u}(\omega,\delta^*)\setminus\{u_0\}
\end{split}
\end{eqnarray*}

\item[{\rm (3)}] If ${\rm dim}V^c(E)=1$ and $\dim V^u(E)$ is even, then
\begin{eqnarray*}
\begin{split}
 %& z(u(\cdot)-u_0(\cdot))\geq N_u+3\quad \text{for }u\in M^s(\omega,\delta^*)\setminus\{u_0\}\\
 & z(u(\cdot)-u_0(\cdot))\geq N_u+2\quad \text{for }u\in M^{s}(\omega,\delta^*)\setminus\{u_0\},\\
 & z(u(\cdot)-u_0(\cdot))=N_u\quad \text{for }u\in M^{c}(\omega,\delta^*)\setminus\{u_0\},\\
 & z(u(\cdot)-u_0(\cdot))\leq N_u\quad \text{for }u\in M^{u}(\omega,\delta^*)\setminus\{u_0\}.\\
 %& z(u(\cdot)-u_0(\cdot))\geq N_u-1\quad \text{for }u\in M^{u}(\omega,\delta^*)\setminus\{u_0\}
\end{split}
\end{eqnarray*}
\end{itemize}
\end{lemma}
\begin{proof}
  See \cite[Corollary 3.5]{SWZ}.
\end{proof}

\begin{lemma}\label{hyperbolic1}
\begin{description}
\item[{\rm (i)}] Assume that $0\notin \sigma (E)$. Then, for $(u_1,g)$, $(u_2,g)\in E$ with $\|u_1-u_2\|\ll1$, one has $M^s(u_1,g,\delta^*)\cap M^u(u_2,g,\delta^*)\neq \emptyset$ and $M^u(u_1,g,\delta^*)\cap M^s(u_2,g,\delta^*)\neq \emptyset$.

\item[{\rm (ii)}]Assume that $0\in \sigma (E)$. Then, for $(u_1,g)$, $(u_2,g)\in E$ with $\|u_1-u_2\|\ll1$, one has $M^{cs}(u_1,g,\delta^*)\cap M^u(u_2,g,\delta^*)\neq \emptyset$ and $M^{s}(u_1,g,\delta^*)\cap M^{cu}(u_2,g,\delta^*)\neq \emptyset$.
\end{description}
\end{lemma}
\begin{proof}
  See \cite[Lemma 3.7]{SWZ}.
\end{proof}

\begin{lemma}\label{spatial-hom}
  Assume that $M\subset X\times H(f)$ be a spatially-homogeneous minimal set of \eqref{equation-lim2} with $\dim V^u(M)>0$. Then $\dim V^c(M)$ must be even and $\dim V^u(M)$ be odd.
\end{lemma}
\begin{proof}
   Given any $\omega=(c_g,g)\in M$($c_g$ is a constant because $M$ is spatially-homogeneous), the variational equation associated with \eqref{equation-lim1} at $(c_g,g)$ turns out to be
\begin{equation}\label{spatial-variation}
  v_t=v_{xx}+a(t)v_x+b(t)v,\quad x\in S^1,
\end{equation} where $a(t)=g_p(t,\varphi(t,\cdot;c_g,g),0)$, $b(t)=g_u(t,\varphi(t,\cdot;c_g,g),0)$. Let $w(t,x)=v(t,x+c(t))e^{-\int_0^tb(s)ds}$ with $\dot{c}(t)=-a(t)$, then \eqref{spatial-variation} is transformed into $w_t=w_{xx}.$ Note that $w_t=w_{xx}$ possesses the simplest ``sin-cos"-mode eigenfunctions as $w_k(t,x)=e^{-k^2t}\sin kx,e^{-k^2t}\cos kx$ associated with the same eigenvalue $\lambda_k=-k^2$, $k=0,1,\cdots.$ Then it yields that $v_k(t,x)=e^{-k^2t+\int_0^tb(s)ds}\sin k (x-c(t)),e^{-k^2t+\int_0^tb(s)ds}\cos k(x-c(t))$ are the Floquet solutions of (4.9) for $k\in \mathbb{N}$. Since $\dim V^u(\omega)>0$, one has that $\dim V^c(M)$ should be even and $\dim V^u(M)$ be odd.
\end{proof}

\begin{lemma}\label{centerodd}
Assume that $0\notin \sigma(E)$ and $\dim V^u(E)\neq 0$. Then $\dim V^u(E)$ must be odd.
\end{lemma}

\begin{proof}
  Let $M\subset E$ be a minimal set of \eqref{equation-lim2}. Then by Remark \ref{invari-space}, we have $\dim V^u(M)\geq \dim V^u(E)$ and $\dim V^c(M)\leq \dim V^c(E)$. Since $\dim V^c(E)=0$, one has $\dim V^c(M)=0$ and $\dim V^u(M)=\dim V^u(E)$, which means $M$ is hyperbolic and hence $M$ is spatially-homogeneous (by Lemma \ref{hyper-minimal}). Then Lemma \ref{spatial-hom} directly implies that $\dim V^u(E)$ is odd.
\end{proof}

\begin{lemma}\label{hyperbolic2}
  If $0\notin \sigma(E)$, then $E$ does not contain any two sided proximal pair.
\end{lemma}
\begin{proof}
  Suppose that there is a two sided proximal pair $\{(u_1,g),(u_2,g)\}$ in $E$. By Lemma \ref{difference-lapnumber}, one may assume that $t_0\in\mathbb{R}$ is such that $\varphi(t_0,\cdot;u_1,g)-\varphi(t_0,\cdot;u_2,g)$ has only simple zeros on $S^1$. Then it follows from Lemma \ref{zero-cons-local}, there is an $\epsilon_0>0$ such that for any $v\in X$ with $\|v\|<\epsilon_0$, $\varphi(t_0,\cdot;u_1,g)-\varphi(t_0,\cdot;u_2,g)+v(\cdot)$ has only simple zeros on $S^1$ and
\begin{equation}\label{localconstant}
  z(\varphi(t_0,\cdot;u_1,g)-\varphi(t_0,\cdot;u_2,g)+v(\cdot))=z(\varphi(t_0,\cdot;u_1,g)-\varphi(t_0,\cdot;u_2,g)).
\end{equation}
Let $\{t_n\}$, $\{s_n\}$ with $t_n\to \infty$, $s_n\to -\infty$ be such that
\begin{equation}
\begin{split}
&\|\varphi(t_n,\cdot;u_1,g)-\varphi(t_n,\cdot;u_2,g)\|\to 0,\\
&\|\varphi(s_n,\cdot;u_1,g)-\varphi(s_n,\cdot;u_2,g)\|\to 0
\end{split}
\end{equation}
as $n\to \infty$. By Lemma \ref{hyperbolic1}(i), there are $u^n_{+}\in M^s(\varphi(t_n,\cdot; u_1,g),g\cdot t_n,\delta^*)\cap M^u(\varphi(t_n,\cdot; u_2,g),g\cdot t_n,\delta^*)$ and $u^n_{-}\in M^s(\varphi(s_n,\cdot; u_1,g),g\cdot s_n,\delta^*)\cap M^u(\varphi(s_n,\cdot; u_2,g),g\cdot s_n,\delta^*)$ for $n$ sufficiently large. Then, $\varphi(s,\cdot;u^n_{+},g\cdot t_n)$ exists for $s<0$ and $\varphi(t,\cdot;u^n_{-},g\cdot s_n)$ exists for $t>0$. Thus by \eqref{exponen-decrea},
\begin{equation}\label{estimate1}
\begin{split}
  &\|\varphi(s,\cdot;u^n_{+},g\cdot t_n)-\varphi(s,\cdot;\varphi(t_n,\cdot;u_2,g),g\cdot t_n)\|\\
  &\leq Ce^{\frac{\alpha s}{2}}\|u^n_+-\varphi(t_n,\cdot;u_2,g)\|\\
\end{split}
\end{equation}
and
\begin{equation}\label{estimate2}
\begin{split}
  &\|\varphi(t,\cdot;u^n_{-},g\cdot s_n)-\varphi(t,\cdot;\varphi(s_n,\cdot;u_1,g),g\cdot s_n)\|\\
  &\leq Ce^{-\frac{\alpha t}{2}}\|u^n_--\varphi(s_n,\cdot;u_1,g)\|
\end{split}
\end{equation}
for any $s\leq 0\leq t$. Note that
\begin{equation}\label{equality}
\begin{split}
&\varphi(t_0-t_n,\cdot;\varphi(t_n,\cdot;u_2,g),g\cdot t_n)=\varphi(t_0,\cdot;u_2,g)\\
&\varphi(t_0-s_n,\cdot;\varphi(s_n,\cdot;u_1,g),g\cdot s_n)=\varphi(t_0,\cdot;u_1,g).
\end{split}
\end{equation}
By \eqref{estimate1}, \eqref{estimate2} and \eqref{equality}, there is $n_0$ sufficiently large such that $s_{n_0}<t_0< t_{n_0}$ and
\begin{equation}
\begin{split}
&\|\varphi(t_0-t_{n_0},\cdot;u^{n_0}_{+},g\cdot t_{n_0})-\varphi(t_0,\cdot;u_2,g)\|<\epsilon_0,\\
&\|\varphi(t_0-s_{n_0},\cdot;u^{n_0}_{-},g\cdot t_{n_0})-\varphi(t_0,\cdot;u_1,g)\|<\epsilon_0.
\end{split}
\end{equation}
Since $u^{n_0}_+\in  M^s(\varphi(t_{n_0},\cdot; u_1,g),g\cdot t_{n_0},\delta^*)$, Lemma \ref{centerodd} and Lemma \ref{zerocenter}(1) imply that
\begin{equation}\label{unstable-control1}
  z(\varphi(t_{n_0},\cdot; u_1,g)-u^{n_0}_+)\geq N_u.
\end{equation}
Hence,
\begin{equation}\label{unstable-control2}
\begin{split}
  &z(\varphi(t_0,\cdot;u_1,g)-\varphi(t_0-t_{n_0},\cdot;u^{n_0}_{+},g\cdot t_{n_0}))\\
  &\geq z(\varphi(t_{n_0},\cdot; u_1,g)-u^{n_0}_+)\geq N_u.
\end{split}
\end{equation}
Combining  \eqref{localconstant}, \eqref{equality} and \eqref{unstable-control2}, one has
\begin{equation*}
\begin{split}
&z(\varphi(t_0,\cdot;u_1,g)-\varphi(t_0,\cdot;u_2,g)-(\varphi(t_0-t_{n_0},\cdot;u^{n_0}_{+},g\cdot t_{n_0})-\varphi(t_0,\cdot;u_2,g)))\\
&=z(\varphi(t_0,\cdot;u_1,g)-\varphi(t_0-t_{n_0},\cdot;u^{n_0}_{+},g\cdot t_{n_0}))\geq N_u.
\end{split}
\end{equation*}
Similarly, one can also obtain that
\begin{equation*}
\begin{split}
&z(\varphi(t_0,\cdot;u_1,g)-\varphi(t_0,\cdot;u_2,g)-(\varphi(t_0,\cdot;u_1,g)-\varphi(t_0-s_{n_0},\cdot;u^{n_0}_{-},g\cdot t_{n_0})))\\
&=z(\varphi(t_0-s_{n_0},\cdot;u^{n_0}_{-},g\cdot t_{n_0})-\varphi(t_0,\cdot;u_2,g))\leq N_u-2,
\end{split}
\end{equation*}
a contradiction. Thus, we have completed the proof.
\end{proof}

Now we are ready to prove Theorem \ref{hyperbolic0}.

\begin{proof}[Proof of Theorem \ref{hyperbolic0}]
 Suppose that $\omega(u_0,g_0)$ is not minimal. Then either case (ii) or case (iii) in Theorem \ref{structure} can hold.

(1) If $\omega(u_0,g_0)=M_1\cup M_{11}$ as in Theorem \ref{structure}(ii), then $M_1$ is spatially-homogeneous and a $1$-cover of $H(f)$
(see Lemma \ref{hyper-minimal}). So, fix any
$(u_{11},g)\in M_{11}$ and $(u_1,g)=M_1\cap p^{-1}(g)$. It then follows from Theorem \ref{structure}(ii) that $(u_1,g)$, $(u_{11},g)$ is a two sided proximal pair, which contradicts Lemma \ref{hyperbolic2}.

(2) If $\omega(u_0,g_0)=M_1\cup M_2\cup M_{12}$  as in Theorem \ref{structure}(iii), then  $M_1$, $M_2$ are $1$-covers of $H(f)$
 (see Lemma \ref{hyper-minimal} again). Fix any $(u_{12},g)\in M_{12}$ and $(u_i,g)=M_i\cap p^{-1}(g)$ for $i=1,2$. For simplicity, we just assume that $\omega(u_{12},g)\cap M_1\neq \emptyset$ and $\alpha(u_{12},g)\cap M_2\neq \emptyset$. Recall that $M_1,M_2$ are $1$-covers. Then Lemma \ref{hyperbolic2} implies that $\a(u_{12},g)\cap M_1=\emptyset$ and $\omega(u_{12},g)\cap M_2=\emptyset$ (Otherwise, $\{(u_1,g),(u_{12},g)\}$ or $\{(u_2,g),(u_{12},g)\}$  forms a two sided proximal pairs with $(u_1,g)\in M_1$ (resp. $(u_2,g)\in M_2$), a contradiction).

We assert that $\Pi^{t}(u_{12},g)-\Pi^{t}(u_{1},g)\to 0$ as $t\to\infty$. Suppose on the contrary that there is a sequence $t_n\to \infty$ such that $\Pi^{t_n}(u_{12},g)\to (u^*_{12},g^*)\ne (u^*_1,g^*)\leftarrow\Pi^{t_n}(u_{1},g)$. Observing that $\omega(u_{12},g)\cap M_2=\emptyset$ and $M_1$ is a $1$-cover, one has $(u^*_{12},g^*)\notin M_1\cup M_2$. Hence, $(u^*_{12},g^*)\in M_{12}$; and moreover, $(\alpha(u^*_{12},g^*)\cup\omega(u^*_{12},g^*))\cap M_2=\emptyset$. Then Theorem \ref{structure}(iii) implies that $\omega(u^*_{12},g^*)\cap M_1\ne \emptyset$ and $\a(u^*_{12},g^*)\cap M_1\ne \emptyset$. In other words, $\{(u^*_{12},g^*),(u^*_1,g^*)\}$ forms a two sided proximal pair, which contradicts Lemma \ref{hyperbolic2}. Thus, we have completed the proof of assertion. Similarly, one may also obtain that $\Pi^{t}(u_{12},g)-\Pi^{t}(u_{2},g)\to 0$ as $t\to -\infty$.

Therefore, by the invariant manifold theory, $\Pi^t(u_{12},g)\in M^s(\Pi^t(u_{1},g),\delta^*)$ for $t\gg 1$. By Lemma \ref{difference-lapnumber}(a) and Lemma \ref{zerocenter}(1), one has
 \begin{equation}\label{E:lower-bdd-1}
   z(\varphi(t,\cdot;u_{1},g)-\varphi(t,\cdot;u_{12},g))\geq N_u, \quad t\in \mathbb{R}^1.
 \end{equation}
  Recall that $\Pi^{t}(u_{12},g)-\Pi^{t}(u_{2},g)\to 0$ as $t\to -\infty$.
We now turn to show that
 \begin{equation}\label{E:lower-bdd-1-1}
   z(\varphi(t,\cdot;u_{1},g)-\varphi(t,\cdot;u_2,g))
   =z(\varphi(t,\cdot;u_{1},g)-\varphi(t,\cdot;u_{12},g)),\, \text{for}\ t\,\text{ sufficiently negative}.
 \end{equation}
In fact, by Lemma \ref{constant-on-twosets}, we know that there is $N_0\in \mathbb{N}$ such that $z(\varphi(t,\cdot;u_{1},g)-\varphi(t,\cdot;u_2,g))=N_0$ for all $t\in \mathbb{R}$.  The compactness of $M_1$ and $M_2$ implies that there exist $t_n\to -\infty$ and $(\tilde u_i,\tilde g)\in M_i$ such that $\Pi^{t_n}(u_i,g)\to(\tilde u_i,\tilde g)$ as $n\to \infty$($i=1,2$), which also means that $\Pi^{t_n}(u_{12},g)\to(\tilde u_2,\tilde g)$ as $n\to\infty$. So by Lemma \ref{zero-cons-local}, $z(\tilde u_1-\tilde u_2)=z(\varphi(t_n,\cdot;u_{1},g)-\varphi(t_n,\cdot;u_{12},g))$ for $n\gg 1$. Furthermore, there exists $m\in\mathbb{N}$ such that for all $n>m$, one has
  \begin{equation}\label{neg-const}
    z(\varphi(t_n,\cdot;u_{1},g)-\varphi(t_n,\cdot;u_{12},g))=N_0.
  \end{equation}
Hence, by Lemma \ref{difference-lapnumber}(a), $z(\varphi(t,\cdot;u_1,g)-\varphi(t,\cdot;u_{12},g)= N_0$ for $t$ sufficiently negative. Consequently, $z(\varphi(t,\cdot;u_1,g)-\varphi(t,\cdot;u_{12},g)\le N_0$ for $t\in \mathbb{R}$. Combining with \eqref{neg-const}, this implies \eqref{E:lower-bdd-1-1} directly.

By virtue of \eqref{E:lower-bdd-1}-\eqref{E:lower-bdd-1-1}, it follows that
 \begin{equation}\label{big-centernumber}
 \begin{split}
   &z(\varphi(t,\cdot;u_{1},g)-\varphi(t,\cdot;u_2,g))\\
   =&z(\varphi(t,\cdot;u_{1},g)-\varphi(t,\cdot;u_{12},g))\geq N_u >0,\quad \text{for}\ t\,\text{ sufficiently negative}.
 \end{split}
 \end{equation}
On the other hand, since $M_1$ and $M_2$ is spatially-homogeneous, it is easy to that $z(\varphi(t,\cdot;u_1,g)-\varphi(t,\cdot;u_2,g))=0$ for all $t\in \mathbb{R}^1$, a contradiction.

Thus, both (ii) and (iii) in Theorem \ref{structure} can not happen, which implies that $\omega(u_0,g_0)$ is minimal; and hence, a spatially-homogeneous $1$-cover of $H(f)$.
\end{proof}

\section{$\omega$-limit sets with $1$-dimensional center spaces}
In this section,  we focus on the structure of general $\omega$-limit sets with nontrivial center spaces. To the best of our knowledge, this question has been hardly studied for both separated boundary and periodic boundary conditions (see some related results for the minimal sets \cite{SWZ,ShenYi-JDDE96}). Here we will consider the structure of $\omega(u_0,g_0)$ with $1$-dimensional center space. To be more specific, we have the following theorem, which seems to be the first attempt to tackle this question.

\begin{theorem}\label{norma-hyper}
Assume that $f(t,u,u_x)=f(t,u,-u_x)$.
  Let $\O=\omega(u_0,g_0)$ be the $\omega$-limit set with $\dim V^c(\Omega)=1$. Then the following hold.
  \begin{itemize}
 \item[{ \rm (i)}] If ${\rm dim}V^u(\Omega)>0$, then  $\O$ is a spatially-inhomogeneous $1$-cover of $H(f)$;

 \item[{ \rm (ii)}]  If ${\rm dim}V^u(\Omega)=0$, then $\Omega$ is spatially-homogeneous.
\end{itemize}
\end{theorem}

We first prove Theorem \ref{norma-hyper}(ii).

\begin{proof}[Proof of Theorem \ref{norma-hyper}(ii)]
 Assume  that $\dim V^u(\Omega)=0$. We write the Sacker-Sell spectrum $\sigma(\O)=\cup_{k=0}^\infty I_k$, where $I_k=[a_k,b_k]$ ($a_0\leq 0\leq b_0$). Suppose that there is $\tilde \omega=(\tilde u,\tilde g)\in \O$ such that $\tilde u(x)$ is spatially-inhomogeneous. Then $\varphi_x(t,\cdot;\tilde u,\tilde g)$ is a nontrivial solution of the linearized equation \eqref{associate-eq}. Moreover, since $\tilde{\omega}=(\tilde u,\tilde g)\in \O$, one has $\|\varphi_x(t,\cdot;\tilde u,\tilde g)\|$ is bounded for all $t\in \mathbb{R}$. So, \eqref{twoside-estimate} implies that $\tilde u_x\in V^c(\tilde \omega)$. Together with $\dim V^c(\tilde \omega)=1$, it yields that $V^c(\tilde \omega)=\mathrm{span}\{\tilde u_x\}$. Let $\lambda=\frac{1}{2}(a_0+b_1)<0$. Then $\Psi_{\lambda}(t,\omega)=e^{-\lambda t}\Psi(t,\omega)$ admits an exponential dichotomy as $X=V^s(\omega)\oplus V^c(\omega)$ over $\O$, i.e., there exist $K_1>0$ and $\alpha_1>0$ such that
\begin{equation}\label{expone-dichot1}
\begin{split}
  &\|\Psi_{\lambda}(t,\omega)u^s\|\leq K_1e^{-\alpha_1 t}\|u^s\|, \ t\geq 0,\quad\forall u^s\in V^s(\omega),\\
  &\|\Psi_{\lambda}(t,\omega)u^c\|\leq K_1e^{\alpha_1 t}\|u^c\|, \ t\leq 0,\quad\forall u^c\in V^c(\omega).
\end{split}
\end{equation}
Recall that $\Psi_{\lambda}(t,\omega)\mid_{V^c(\omega)}$ is an isomorphism, then one has
\begin{equation}\label{center-estimate}
  K^{-1}_1e^{\alpha_1t}\|u^c\|\leq \|\Psi_{\lambda}(t,\omega)u^c\|\ \text{for any} \ t\geq 0,\ u^c\in V^c(\omega), \ \omega\in\O.
\end{equation}
Together with \eqref{expone-dichot1} and \eqref{center-estimate}, we have
\begin{equation}\label{seperation}
  \|\Psi(t,\omega)u^s\|\leq K_1^{2}e^{-2\alpha_1t}\|\Psi(t,\omega)u^c\|,\quad t\geq 0,\ \omega\in \O,
\end{equation}
for any $u^s\in V^s(\omega)$ and $u^c\in V^c(\omega)$ with $\|u^s\|=\|u^c\|=1$. On the other hand, by the exponentially separated property of the strongly monotone skew-product semiflows (see, e.g. \cite[p.38]{Shen1998}), $X=X_1(\omega)\oplus X_2(\omega)$ where $X_i(\omega)$($i=1,2$) vary continuously in $\omega\in \O$ and $X_1(\omega)=\mathrm{span}\{v(\omega)\}$ with $v(\omega)\in \mathrm{Int} X^+$ and $\|v(\omega)\|=1$. Moreover, there exist $\alpha_2>0$ and $K_2>0$ such that for any $u_2\in X_2(\omega)$ with $\|u_2\|=1$,
\begin{equation}\label{continuous-separation}
  \|\Psi(t,\omega)u_2\|\leq K_2e^{-\alpha_2t}\|\Psi(t,\omega)v(\omega)\| \quad(t>0).
\end{equation}
Let $\omega=\tilde \omega$. Then $X=X_1(\tilde \omega)\oplus X_2(\tilde\omega)$ and $X=V^c(\tilde\omega)\oplus V^s(\tilde\omega)$ satisfying \eqref{seperation} and \eqref{continuous-separation}. By the uniqueness of the exponential separation on $X$ (see, e.g. \cite[Theorem 3.2.3]{Mierczynski}),  one has $X_1(\tilde \omega)=V^c(\tilde \omega)$ and $X_2(\tilde \omega)=V^s(\tilde\omega)$. So, $\tilde u_x\in \mathrm{Int} X^+$, which implies that $\tilde u(x)$ is a strictly monotone function for $x\in S^1$. This contradicts to the periodic boundary condition. Thus, we have completed the proof of Theorem 5.1(ii).
\end{proof}

We now make some preparation for proving Theorem \ref{norma-hyper}(i). For given $u\in X$ and $a\in S^1$, we define $\sigma_a u=u(\cdot+a)$ and let
$$
\Sigma u=\{\sigma_au\,|\, a\in S^1\},
$$
we also let $m(u)=\max_{x\in S^1}u(x)$ be the maximal value of $u$ on $S^1$. For any $a\in S^1$, if $\varphi(t,\cdot;u,g)$ is a classical solution of \eqref{equation-lim1}, then it is easy to check that $\sigma_a\varphi(t,\cdot;u,g)$ is a classical solution of \eqref{equation-lim1}. Moreover, the uniqueness of solution ensures the translation invariance, that is, $\sigma_a\varphi(t,\cdot;u,g)=\varphi(t,\cdot;\sigma_au,g)$.

\begin{lemma}\label{equivalence}
Let $M\subset X\times H(f)$ be a minimal set of \eqref{equation-lim2}. Assume that $\dim V^c(M)=1$. Then $M$ is spatially-homogeneous if and only if $\dim V^u(M)=0$.
\end{lemma}
\begin{proof}
The necessity directly part follows from Lemma \ref{spatial-hom} (Otherwise, $\dim V^u(\omega)\ne 0$ will imply that $\dim V^c(M)$ must be even, a contradiction to $\dim V^c(M)=1$).

The sufficiency part is a direct corollary of Theorem \ref{norma-hyper}(ii).
\end{proof}

In the remaining of this section, we always assume that $\dim V^u(\Omega)>0$ and $\dim V^c(\Omega)=1$.

\begin{lemma}\label{homogeneous}
  Let $\Omega=\omega(u_0,g_0)$ be an $\omega$-limit set of \eqref{equation-lim2} and $M$ be any minimal set in $\O$. Assume that  $\dim V^c(\Omega)=1$ and $\dim V^u(\Omega)>0$. Then $\dim V^c(M)\le 1$.  Furthermore,

 {\rm (a)} When $\dim V^c(M)=1$, $M$ is a spatially-inhomogeneous $1$-cover of $H(f)$; and moreover, there is $\delta^*>0$ such that $M^c(\omega,\delta^*)\subset \Sigma u$ for any $\omega=(u,g)\in M$.

 {\rm (b)} When $\dim V^c(M)=0$, $M$ is a spatially-homogeneous $1$-cover of $H(f)$; and moreover, one has
  \begin{equation*}
\left\{
\begin{split}
 &\dim V^u(M)=\dim V^u(\Omega)\textnormal{ and }\mathrm{codim} V^s(M)=\mathrm{codim}V^s(\Omega)-1,\,\quad\,\text{ if }{\rm dim}V^u(\Omega)\text{ is odd;}\\
 &\dim V^u(M)=\dim V^u(\Omega)+1\textnormal{ and } \mathrm{codim} V^s(M)=\mathrm{codim}V^s(\Omega),\quad\,\text{ if }{\rm dim}V^u(\Omega)\text{ is even.}
\end{split}\right.
\end{equation*}
\end{lemma}
\begin{proof}
 By Remark \ref{invari-space}, we have $\dim V^c(M)\le 1$. Furthermore, case (b) is directly from Lemma \ref{hyper-minimal} and Lemma \ref{centerodd}.

Now we focus on proving (a). Recall that $M$ is minimal. Then,  $M$ is either spatially-inhomogeneous; or otherwise, $M$ is spatially-homogeneous. Since $\dim V^c(M)=1$, Lemma \ref{spatial-hom} entails that $M$ is spatially-inhomogeneous; and hence, one can immediately utilize \cite[Theorem 4.2(2)]{SWZ} to obtain that $M$ is a $1$-cover of $H(f)$.

It remains to show that $M^c(\omega,\delta^*)\subset \Sigma u$ for any $\omega=(u,g)\in M$. To this end, we note that, due to the $2\pi$-periodic boundary condition, any element of $M$ will possess a spatial-period of $2\pi$, that is, $u(\cdot)=u(\cdot+2\pi)$ for any $(u,g)\in M$. Let $L\in (0,2\pi]$ be the smallest spatial-period of some element in $M$. Then the minimality of $M$ yields that $L$ is the smallest spatial-period of any element in $M$. Choose $\delta\in (0,L)$ so small that
$$\|\varphi(t,\cdot;\sigma_au,g)-\varphi(t,\cdot;u,g)\|<\delta^*, \,\textnormal{ for any }\abs{a}<\delta\,\,\textnormal{and }t\in \mathbb{R},$$ where $\delta^*>0$ is defined in Lemma \ref{invari-mani}. From Remark \ref{stable-leaf}(2)-(3), it then follows that $\sigma_au\in M^{cs}(\omega,\delta^*)\cap M^{cu}(\omega,\delta^*)$=$M^{c}(\omega,\delta^*)$ for all $\abs{a}<\delta$.

 Since $\dim V^c(M)=1$, it induces that $M^c(\omega,\delta^*)\subset \Sigma u$. Thus, we have proved the Lemma.

\end{proof}

\vskip 2mm
Recall that
\begin{equation*}
N_u=\left\{
\begin{split}
 &{\dim} V^u(\Omega),\,\quad\,\,\,\text{ if }{\rm dim}V^u(\Omega)\text{ is even,}\\
 &{\rm dim}V^u(\Omega)+1,\,\text{ if }{\rm dim}V^u(\Omega)\text{ is odd.}
\end{split}\right.
\end{equation*}

\begin{lemma}\label{norma-hyperb2}
 Assume that $\dim V^u(\Omega)>0$ and $\dim V^c(\Omega)=1$. Let $M\subset\Omega$ be a minimal set. Then for any $(u_1,g)\in M$ and $(u_2,g)\in \Omega\setminus M$, the pair $\{(u_1,g),(u_2,g)\}$ can not be two sided proximal.
\end{lemma}
\begin{proof}
By virtue of Lemma \ref{homogeneous}, $M$ is a $1$-cover. Moreover, it is either (i) spatially-inhomogeneous (with ${\rm dim}V^c(M)=1$), or (ii) spatially-homogeneous (with ${\rm dim}V^c(M)=0$).

Case {\rm (i)}. $M$ is spatially-inhomogeneous with ${\rm dim}V^c(M)=1$. Suppose there are $(u_1,g)\in M$ and $(u_2,g)\in \O\setminus M$ such that $\{(u_1,g),(u_2,g)\}$ forms a two sided proximal pair.  Then it is easy to see that $u_2\notin\Sigma u_1$ (Otherwise, $u_2=\sigma_{a_0}u_1$ for some $a_0\in S^1\setminus \{e\}$ and since $\{(u_1,g),(u_2,g)\}$ is  two sided proximal pair, there exist a sequence $\{t_n\}$ and $(u^*,g^*)\in M$ such that $\Pi^{t_n}(u_1,g)\to (u^*,g^*)$ and $\Pi^{t_n}(\sigma_{a_0}u_1,g)\to (u^*,g^*)$ as $n\to \infty$. This entails that $\sigma_{a_0}u^*=u^*$. By the minimality of $M$, one has $\sigma_{a_0}u_1=u_1$. So, $u_2=u_1$, a contradiction.)

Now choose two sequences $t_n\to \infty$ and $s_n\to -\infty$ such that
\begin{equation}\label{posi-proxi2}
\begin{split}
 \Pi^{t_n}(u_1,g)\to(u^*,g^*) \,\,\text{ and }\,\,
\Pi^{t_n}(u_2,g)\to(u^*,g^*)
\end{split}
\end{equation}
and
\begin{equation}\label{nega-proxi2}
\begin{split}
 \Pi^{s_n}(u_1,g)\to(u^{**},g^{**}) \,\,\text{ and }\,\,
 \Pi^{s_n}(u_2,g)\to(u^{**},g^{**})
\end{split}
\end{equation}
as $n\to\infty$. Clearly, $(u^*,g^*), (u^{**},g^{**})\in M$. It also follows from \cite[Corollary 3.9]{SWZ} that $z(u^*-\sigma_au^*)=N_u=z(u^{**}-\sigma_au^{**})$ for any $a\in S^1$ with $\sigma_ au^*\neq u^*$ and $\sigma_ au^{**}\neq u^{**}$.

Since $u_2\notin \Sigma u_1$, we claim that $z(\varphi(t,\cdot;u_2,g)-\varphi(t,\cdot;\sigma_au_1,g))\geq N_u$ for all $t\in \mathbb{R}^1$ and $a\in S^1$. Otherwise, by Lemma \ref{difference-lapnumber}, there is $a_0\in S^1$, $T\in \mathbb{R}^1$ and $N\in\mathbb{N}$ with $N<N_u$ such that $z(\varphi(t,\cdot;u_2,g)-\varphi(t,\cdot;\sigma_{a_0}u_1,g))=N$ for all $t\geq T$.
Then the continuity of $z(\cdot)$ implies that there is a small neighborhood $B(a_0)$ of $a_0$ such that
$z(\varphi(T,\cdot;u_2,g)-\varphi(T,\cdot;\sigma_{a}u_1,g))=N$ for any $a\in B(a_0)$. Since $u^*$ is spatially-inhomogeneous, one can always choose some $a_1\in B(a_0)$ such that  $\sigma_{a_1}u^*\neq u^*$.
By Lemmas \ref{difference-lapnumber} and \ref{sequence-limit}, one has $z(u^*-\sigma_{a_1}u^*)\leq N<N_u$, a contradiction. Thus, we have proved this claim.

By a similar deduction with the sequence $\{s_n\}$, one can also get $z(\varphi(t,\cdot;u_2,g)-\varphi(t,\cdot;\sigma_au_1,g))\leq N_u$ for all $t\in \mathbb{R}^1$ and $a\in S^1$. So, $$z(\varphi(t,\cdot;u_2,g)-\varphi(t,\cdot;\sigma_a u_1,g))=N_u,\,\, \text{ for all }t\in \mathbb{R}\text{ and }a\in S^1.$$  In particular, by Lemma \ref{zero-cons-local} and the compactness of $S^1$, there exists $\delta>0$ (independent of $a\in S^1$) such that \begin{equation}\label{case-A-const-1}
z(u_2-\sigma_a u_1+v)=N_u
\end{equation}
for any $a\in S^1$ and $v\in X$ with $\|v\|<\delta$.

We now consider two cases separately: $\dim V^u(\Omega)$ is even; or $\dim V^u(\Omega)$ is odd. In the following, we will prove case (i) under the assumption that $\dim V^u(\Omega)$ is even. The proof is analogous under the assumption of $\dim V^u(\Omega)$ being odd.

Using \eqref{posi-proxi2} and Lemma \ref{hyperbolic1}(ii) we obtain that there exists some $v\in M^u(\varphi(t_n,\cdot;u_2,g),\delta^*)\cap M^{cs}(\varphi(t_n,\cdot;u_1,g),\delta^*)$ for $t_n$ sufficiently large.
We now assert that $v\notin M^c(\varphi(t_n,\cdot;u_1,g),\delta^*)$. Otherwise, since $\dim V^c(M)=\dim V^c(\Omega)=1$, by Lemma \ref{homogeneous}(a), one has $v=\sigma_{a} \varphi(t_n,\cdot;u_1,g)$ for some $a\in S^1$. Noticing that $v\in M^u(\varphi(t_n,\cdot;u_1,g),\delta^*)$, then
\begin{equation}\label{backward-contracting-1}
\|\sigma_{a}u_1-u_2\|=\norm{\varphi(-t_n,\cdot;v,g\cdot t_n)-u_2}\leq Ce^{-\frac{\alpha}{2} t_n}\|v-\varphi(t_n,\cdot;u_2,g)\|\le C\delta^*e^{-\frac{\alpha}{2} t_n}.
\end{equation}
Let $\epsilon_0=|m(u_1)-m(u_2)|>0$ with $m(u_i)=\max_{x\in S^1} u_i(x)$ for $i=1,2$ (note that $\epsilon_0>0$ is due to $u_2\notin\Sigma u_1$ and \cite[Corollary 3.9 (iii)]{SWZ}). Since $X\hookrightarrow C^1(S^1)$, it is not difficult to see that $\|\sigma_{a}u_1-u_2\|\geq C_0|m(u_1)-m(u_2)|=C_0\epsilon_0$ for some constant $C_0>0$. As a consequence, by letting $t_n$  be large enough, one has $\|\sigma_{a}u_1-u_2\|<\mathrm{min}\{\delta^*,C_0\epsilon_0\}$. Thus, we have obtained a contradiction. Therefore, $v\notin M^c(\varphi(t_n,\cdot;u_1,g),\delta^*)$.

Recall that $v\in M^{cs}(\varphi(t_n,\cdot;u_1,g),\delta^*)$. Then it follows from the foliation of $M^{cs}(\varphi(t_n,\cdot;u_1,g),\delta^*)$ in Remark \ref{stable-leaf}(4) and Lemma \ref{homogeneous}(a) that there is some $a_0\in S^1$ such that $v\in M^{s}(\sigma_{a_0}\varphi(t_n,\cdot;u_1,g),\delta^*)$ with $\sigma_{a_0}\varphi(t_n,\cdot;u_1,g)\in M^c(\varphi(t_n,\cdot;u_1,g),\delta^*)$. Since $\dim V^u(\Omega)$ is even, Lemma \ref{zerocenter}(3) entails that $z(v-\sigma_{a_0}\varphi(t_n,\cdot;u_1,g))\ge N_u+2$, and hence,
\begin{equation}\label{u-3-Nu-large}
z(\varphi(-t_n,\cdot;v,g\cdot t_n)-\sigma_{a_0}u_1)\geq N_u+2.
\end{equation}
But then, one can deduce from \eqref{backward-contracting-1} that $\|\varphi(-t_n,\cdot;v,g\cdot t_n)-u_2\|<\delta$, for $t_n$ sufficiently large. Here $\delta>0$ is defined in \eqref{case-A-const-1}. As a consequence, by \eqref{case-A-const-1}, one has $z(\varphi(-t_n,\cdot;v,g\cdot t_n)-\sigma_{a_0}u_1)=z(\varphi(-t_n,\cdot;v,g\cdot t_n)-u_2+u_2-\sigma_{a_0}u_1)=z(u_2-\sigma_{a_0}u_1)=N_u$, contradicting \eqref{u-3-Nu-large}.
Thus, we have proved case (i) under the assumption of $\dim V^u(\Omega)$ being even.
\vskip 2mm

Case {\rm (ii)}. $M$ is spatially-homogeneous with ${\rm dim}V^c(M)=0$. Suppose that there are $(u_1,g)\in M$ and $(u_2,g)\in \O\setminus M$ such that $\{(u_1,g),(u_2,g)\}$ forms a two sided proximal pair.
Again, we will prove case (ii) under the assumption that $\dim V^u(\Omega)$ is odd; or  $\dim V^u(\Omega)$ is even, separately.

We first consider the assumption of $\dim V^u(\Omega)$ being odd. By virtue of Lemma \ref{homogeneous}(b), one can choose $\delta^*>0$ so small that
\begin{equation}\label{E:cs=unstable}
M^{u}(\omega,\delta^*)=\tilde M^u(\omega,\delta^*)\,\text{ and }\,M^{cs}(\omega,\delta^*)=\tilde M^{s}(\omega,\delta^*),
\end{equation}
where $\tilde M^u,\tilde M^s$ denote respectively the local unstable and stable manifolds of $\omega\in M$ with respect to the Sacker-Sell spectrum $\sigma(M)$ (see more discussion on \eqref{E:cs=unstable} in Remark \ref{R:uniqueness of W-cs}(i)).

For simplicity, we may assume that $u_2-u_1$ has only simple zeros on $S^1$. Then, by Lemma \ref{zero-cons-local}, there are $\delta>0$ and $N\in\mathbb{N}$, such that for any $v\in X$ with $\|v\|<\delta$,
\begin{equation}\label{E:zero-control}
  z(u_2-u_1+v)=N.
\end{equation}
Since $\{(u_1,g),(u_2,g)\}$ forms a two sided proximal pair, there are two sequences $t_n\to \infty$ and $s_n\to -\infty$ such that
\begin{equation*}
\begin{split}
 \|\varphi(t_n,\cdot;u_1,g)-\varphi(t_n,\cdot;u_2,g)\|\to 0\,\,\text{ and }\,\,
 \|\varphi(s_n,\cdot;u_1,g)-\varphi(s_n,\cdot;u_2,g)\|\to 0, \,\,\text{ as }n\to \infty.
\end{split}
\end{equation*}
By Lemma \ref{hyperbolic1}, one has $M^{u}(\Pi^{t_n}(u_2,g),\delta^*)\cap M^{cs}(\Pi^{t_n}(u_1,g),\delta^*)\neq \emptyset$ and $M^{u}(\Pi^{s_n}(u_2,g),\delta^*)\cap M^{cs}(\Pi^{s_n}(u_1,g),\delta^*)\neq \emptyset$ for $n$ sufficiently large. Then \eqref{E:cs=unstable} entails that
%$$M^{u}(\Pi^{t_n}(u_2,g),\delta^*)\cap \tilde M^{s}(\Pi^{t_n}(u_1,g),\delta^*)\neq \emptyset$$  and
 $$M^{u}(\Pi^{s_n}(u_2,g),\delta^*)\cap \tilde M^{s}(\Pi^{s_n}(u_1,g),\delta^*)\neq \emptyset$$
for $n$ sufficiently large. Choose some $u_n^+\in M^{u}(\Pi^{t_n}(u_2,g),\delta^*)\cap  M^{cs}(\Pi^{t_n}(u_1,g),\delta^*)$ and $u_n^-\in M^{u}(\Pi^{s_n}(u_2,g),\delta^*)\cap \tilde M^{s}(\Pi^{s_n}(u_1,g),\delta^*)$. Let $u_2^*=\varphi(-t_n,\cdot;u_n^+,g)$ and $u_1^*=\varphi(-s_n,\cdot;u_n^-,g)$. Then \eqref{exponen-decrea} implies that
\begin{eqnarray*}\label{A-estimate-forward}
  \|u^*_2-u_2\|&=&\|\varphi(-t_n,\cdot;u_n^{+},g\cdot t_n)-\varphi(-t_n,\cdot;\varphi(t_n,\cdot;u_2,g),g\cdot t_n)\| \notag\\
  &\leq& Ce^{-\frac{\alpha}{2}t_n}\|u_n^+-\varphi(t_n,\cdot;u_2,g)\|
 \end{eqnarray*} and
\begin{eqnarray*}\label{A-estimate-backward}
\|u^*_1-u_1\|&=&\|\varphi(-s_n,\cdot;u_n^{-},g\cdot s_n)-\varphi(-s_n,\cdot;\varphi(s_n,\cdot;u_1,g),g\cdot s_n)\| \notag\\
&\leq& Ce^{\frac{\alpha}{2}s_n}\|u_n^--\varphi(s_n,\cdot;u_1,g)\|.
\end{eqnarray*}
Consequently, one has $\|u^*_2-u_2\|<\delta$ and $\|u^*_1-u_1\|<\delta$ for $n$ sufficiently large, where $\delta$ is from \eqref{E:zero-control}. Thus, $z(u_2-u_1)=z(u^*_2-u_1)\geq z(u^+_n-\varphi(t_n,\cdot;u_1,g))$. Noticing that $u_n^+\in M^{cs}(\Pi^{t_n}(u_1,g),\delta^*)$, by Lemma \ref{zerocenter}(2), $z(u^+_n-\varphi(t_n,\cdot;u_1,g))\geq N_u$, and hence, $z(u_2-u_1)\geq N_u$. On the other hand, $z(u_2-u_1)=z(u_2-u^*_1)\leq z(\varphi(s_n,\cdot;u_2,g)-u^-_n)$. Recall that $u_n^-\in M^{u}(\Pi^{s_n}(u_2,g),\delta^*)$. Then Lemma \ref{zerocenter}(2) implies that $z(\varphi(s_n,\cdot;u_2,g)-u^-_n)<N_u$, which entails that $z(u_2-u_1)<N_u$, a contradiction.

Now we consider the assumption of $\dim V^u(\Omega)$ being even. Also by Lemma \ref{homogeneous}(b), one can choose $\delta^*>0$ small enough such that
\begin{equation}\label{E:cu=unstable}
M^{cu}(\omega,\delta^*)=\tilde M^u(\omega,\delta^*)\,\text{ and }\,M^{s}(\omega,\delta^*)=\tilde M^{s}(\omega,\delta^*).
\end{equation}
(see more discussion on \eqref{E:cu=unstable} also in Remark \ref{R:uniqueness of W-cs}(ii)).
Now we suppose that
\begin{equation*}
\begin{split}
 \|\varphi(t_n,\cdot;u_1,g)-\varphi(t_n,\cdot;u_2,g)\|\to 0\,\,\text{ and }\,\,
 \|\varphi(s_n,\cdot;u_1,g)-\varphi(s_n,\cdot;u_2,g)\|\to 0, \,\,\text{ as }n\to \infty,
\end{split}
\end{equation*}
where $t_n\to \infty$ and $s_n\to -\infty$. By Lemma \ref{hyperbolic1}, for $n$ sufficiently large, one has $M^{cu}(\Pi^{t_n}(u_1,g),\delta^*)\cap M^{s}(\Pi^{t_n}(u_2,g),\delta^*)\neq \emptyset$ and $M^{cu}(\Pi^{s_n}(u_1,g),\delta^*)\cap M^{s}(\Pi^{s_n}(u_2,g),\delta^*)\neq \emptyset$. Using \eqref{E:cu=unstable}, one has

$$\tilde M^{u}(\Pi^{t_n}(u_1,g),\delta^*)\cap  M^{s}(\Pi^{t_n}(u_2,g),\delta^*)\neq \emptyset$$
for $n$ sufficiently large. Now choose $w_n^+\in \tilde M^{u}(\Pi^{t_n}(u_1,g),\delta^*)\cap M^{s}(\Pi^{t_n}(u_2,g),\delta^*)$ and $w_n^-\in M^{cu}(\Pi^{s_n}(u_1,g),\delta^*)\cap M^{s}(\Pi^{s_n}(u_2,g),\delta^*)$. Let $u_1^{**}=\varphi(-t_n,\cdot;w_n^+,g\cdot t_n)$ and $u_2^{**}=\varphi(-s_n,\cdot;u_n^-,g\cdot s_n)$. Then by \eqref{exponen-decrea},
\begin{equation*}\label{B-estimate}
\begin{split}
  &\|u^{**}_1-u_1\|=\|\varphi(-t_n,\cdot;w_n^{+},g\cdot t_n)-\varphi(-t_n,\cdot;\varphi(t_n,\cdot;u_1,g),g\cdot t_n)\|\\
  &\leq Ce^{-\frac{\alpha t_n}{2}}\|w_n^+-\varphi(t_n,\cdot;u_1,g)\|,\\
  &\|u^{**}_2-u_2\|=\|\varphi(-s_n,\cdot;w_n^{-},g\cdot s_n)-\varphi(-s_n,\cdot;\varphi(s_n,\cdot;u_2,g),g\cdot s_n)\|\\
  &\leq  Ce^{\frac{\alpha s_n}{2}}\|w_n^--\varphi(s_n,\cdot;u_2,g)\|
\end{split}
\end{equation*}
Thus, $\|u^{**}_1-u_1\|<\delta$ and $\|u^{**}_2-u_2\|<\delta$ for $n\gg 1$. Then, $z(u_2-u_1)=z(u_2-u^{**}_1)\geq z(\varphi(t_n,\cdot;u_2,g)-w^+_n)$. Since $w^+_n\in M^{s}(\Pi^{t_n}(u_2,g),\delta^*)$, by Lemma \ref{zerocenter}(3), $z(\varphi(t_n,\cdot;u_2,g)-w^+_n)\geq N_u+2$. Consequently, $z(u_2-u_1)\geq N_u+2$. Similarly, one has $z(u_2-u_1)=z(u^{**}_2-u_1)\leq z(w^-_n-\varphi(s_n,\cdot;u_1,g))\leq  N_u$, a contradiction.

Thus, we have completed the proof of this lemma.
\end{proof}
\vskip 2mm

\begin{remark}\label{R:uniqueness of W-cs}
{\rm
(i) Let $\delta^*$ be defined in Lemma \ref{invari-mani}. It deserves to point out that one can choose some smaller $\delta^*_1\in (0,\delta^*)$, if necessary, such that $M^{cs}(\omega,\delta^*_1)=\tilde M^{s}(\omega,\delta^*_1)$, which entails that in our case $M^{cs}(\omega,\delta^*_1)$ is unique (note that in general case $M^{cs}(\omega,\delta^*_1)$ could be not unique). Indeed, choose $\delta_1^*$ so small, one can obtain that
$$\|\varphi(t,\cdot;w,g)-\varphi(t,\cdot;u,g)\|<Ce^{-\beta t}\norm{w-u}<C\delta^*_1<\delta^*, \quad t\ge 0,$$ for any
$w\in \tilde{M}^{s}(\omega,\delta^*_1)$. Then for any $M^{cs}(\omega,\delta^*)$, it follows from Remark \ref{stable-leaf}(2) that $w\in M^{cs}(\omega,\delta^*)$. Note also $\norm{w-u}<\delta^*_1$. Then $w\in M^{cs}(\omega,\delta^*_1)$. This implies that $\tilde{M}^{s}(\omega,\delta^*_1)\subset M^{cs}(\omega,\delta^*_1)$. Now, given any $w\in M^{cs}(\omega,\delta^*_1)$, by the graph-expression of $M^{cs}(\omega,\delta^*_1)$ (as in Lemma \ref{invari-mani}), one can find some $v\in X^{cs}(\omega)$ with $\norm{v}<\delta^*_1$ such that
$w=u+v+h^{cs}(v,\omega)$. Since $X^{cs}(\omega)=X^s(\omega)\oplus X^c(\omega)=\tilde{X}^{s}(\omega)$ for any $\omega\in M$, one has $w_1:=u+v+\tilde{h}^s(v,\omega)\in \tilde{M}^{s}(\omega,\delta^*_1)$. Recall that  $\tilde{M}^{s}(\omega,\delta^*_1)\subset M^{cs}(\omega,\delta^*_1)$. Then $w_1\in  M^{cs}(\omega,\delta^*_1)$. By the graph-property of $M^{cs}(\omega,\delta^*_1)$, it yields that
$h^{cs}(v,\omega)=\tilde{h}^s(v,\omega)$, which implies that $w=w_1\in \tilde{M}^{s}(\omega,\delta^*_1)$. Therefore, $M^{cs}(\omega,\delta^*_1)=\tilde M^{s}(\omega,\delta^*_1)$. Thus, the uniqueness of $\tilde{M}^s(\omega,\delta^*_1)$ immediately implies that $M^{cs}(\omega,\delta^*_1)$ is unique regardless of the choose of the cut-off function.

(ii) Similarly as in the above argument, we can also utilize Remark \ref{stable-leaf}(3) to obtain that one can choose smaller $\delta^*_1\in (0,\delta^*)$ such that $M^{cu}(\omega,\delta^*_1)=\tilde M^{u}(\omega,\delta^*_1)$ is unique regardless of the choose of the cut-off function.
}
\end{remark}
Now we are ready to prove Theorem \ref{norma-hyper}(i).

\begin{proof}[Proof of Theorem \ref{norma-hyper}{\rm (i)}]

Assume that $\dim V^u(\O)>0$. Then Lemma \ref{homogeneous} implies that any minimal set $M\subset \O$ is either a spatially-homogeneous $1$-cover with $\dim V^c(M)=0$, or a spatially-inhomogeneous $1$-cover with $\dim V^c(M)=1$.

 {\it We claim that the $\omega$-limit set $\O$ itself is minimal.} Before we give the proof of this claim, it is clear from \cite[Theorem 4.2(ii)]{SWZ} that this claim implies Theorem \ref{norma-hyper}(i) immediately.

 Now we are focusing on the proof of this claim. By the structure of $\O$ given in Theorem \ref{structure}, to prove this claim, it suffices to show that neither (ii) nor (iii) in Theorem \ref{structure} can hold.

Case (a): $\omega(u_0,g_0)=M_1\cup M_{11}$ as in Theorem \ref{structure}(ii), where $M_1$ is a $1$-cover. Fix any $(u_{11},g)\in M_{11}$ and let $(u_1,g)=M_1\cap p^{-1}(g)$. Then it follows from Theorem \ref{structure}(ii) that $M_1\subset \alpha(u_{11},g)\cap \omega(u_{11},g)$. Since  $M_1$ is an $1$-cover, $\{(u_{11},g),(u_1,g)\}$ forms a two sided proximal pair, contradicting Lemma \ref{norma-hyperb2}. So, Theorem \ref{structure}(ii) cannot happen.

Case (b): $\omega(u_0,g_0)=M_1\cup M_2\cup M_{12}$ as in Theorem \ref{structure}(iii), where $M_1$, $M_2$ are both $1$-covers of $H(f)$. Fix any $(u_{12},g)\in M_{12}$ and let $(u_i,g)=M_i\cap p^{-1}(g)$ for $i=1,2$. Clearly, $u_{12}\ne u_i$ for $i=1,2.$ We will discuss the following three alternatives separately:

(b1) Both $M_1$ and $M_2$ are spatially-inhomogeneous;

(b2) Both $M_1$ and $M_2$ are spatially-homogeneous;

(b3) One is spatially-homogeneous, the other is spatially-inhomogeneous.

For case (b1):
 By Theorem \ref{structure}(iii), we may assume that $\omega(u_{12},g)\cap M_1\neq \emptyset$ and $\alpha(u_{12},g)\cap M_2\neq \emptyset$. This then implies that $\a(u_{12},g)\cap M_1=\emptyset$ and $\omega(u_{12},g)\cap M_2=\emptyset$ (Otherwise, $\{(u_1,g),(u_{12},g)\}$ or $\{(u_2,g),(u_{12},g)\}$ will  form a two sided proximal pair, which contradicts Lemma \ref{norma-hyperb2}). Similarly as in the assertion in the proof of Theorem \ref{hyperbolic0}, we can obtain
 \begin{equation}\label{E:approach-each}
 \Pi^{t}(u_{12},g)-\Pi^{t}(u_{1},g)\to 0,\,\, \text{ as }t\to\infty.
 \end{equation}
 As a consequence, by Remark \ref{stable-leaf}(2), $\varphi(t,\cdot;u_{12},g)\in M^{cs}(\Pi^t(u_{1},g),\delta^*)$ for $t\gg 1$.
 Furthermore, {\it we assert that $\varphi(t,\cdot;u_{12},g)\in M^{s}(\Pi^t(u_{1},g),\delta^*)$ for $t\gg 1$.} In fact, fix any $T\gg 1$ with $\varphi(T,\cdot;u_{12},g)\in M^{cs}(\Pi^T(u_{1},g),\delta^*)$, by the foliation of $M^{cs}(\Pi^T(u_{1},g),\delta^*)$ in Remark \ref{stable-leaf}(4), there exists a $u_c\in M^c(\Pi^T(u_{1},g),\delta^*)$ such that $\varphi(T,\cdot;u_{12},g)\in \bar{M}_s(u_c,\Pi^T(u_{1},g),\delta^*)$. By Lemma \ref{homogeneous}(a), we have $M^c((\Pi^T(u_{1},g),\delta^*)\subset \Sigma(\varphi(T,\cdot;u_1,g))$. So, one can find some $a\in S^1$ such that $u_c=\varphi(T,\cdot;\sigma_au_1,g)$; and hence, we have
 $$\varphi(T,\cdot;u_{12},g)\in \bar{M}_s(\varphi(T,\cdot;\sigma_au_1,g),\Pi^T(u_{1},g),\delta^*).$$ This entails that $\norm{\varphi(t+T,\cdot;u_{12},g)-\varphi(t+T,\cdot;\sigma_au_1,g)}\to 0$ as $t\to \infty$. As a consequence, $\Pi^{t}(u_{12},g)-\Pi^{t}(\sigma_au_{1},g)\to 0$ as $t\to\infty$. Together with \eqref{E:approach-each}, we obtain that $\sigma_a u_{1}=u_{1}$, which implies that $\varphi(T,\cdot;u_{12},g)\in M^{s}(\Pi^T(u_{1},g),\delta^*)$. By the arbitrariness of $T\gg 1$, one has $\varphi(t,\cdot;u_{12},g)\in M^{s}(\Pi^t(u_{1},g),\delta^*)$ for all $t\gg 1$. Thus, we have proved the assertion.

Similarly as in the proof of Lemma \ref{norma-hyperb2}, we consider separately the situation that $\dim V^u(\Omega)$ is even and $\dim V^u(\Omega)$ is odd. We again only deal with the situation of $\dim V^u(\Omega)$ being even.
By Lemma \ref{difference-lapnumber} and Lemma \ref{zerocenter}(3), the assertion above implies that
 \begin{equation}\label{E:normal-hy-1}
   z(\varphi(t,\cdot;u_{12},g)-\varphi(t,\cdot;u_1,g))\geq N_u+2, \quad t\in \mathbb{R}^1,
 \end{equation}
and
 \begin{equation}\label{E:normal-hy-2}
   z(\varphi(t,\cdot;u_{12},g)-\varphi(t,\cdot;u_2,g))\leq N_u, \quad t\in \mathbb{R}^1.
 \end{equation}
On the one hand, by \eqref{E:approach-each} and \eqref{E:normal-hy-2}, similarly as \eqref{E:lower-bdd-1-1} we have
 \begin{equation}\label{E:normal-hy-3}
 \begin{split}
   z(\varphi(t,\cdot;u_{1},g)-\varphi(t,\cdot;u_2,g))
   =z(\varphi(t,\cdot;u_{12},g)-\varphi(t,\cdot;u_2,g))\leq N_u,\quad \text{for}\ t\gg 1.
 \end{split}
 \end{equation}
On the other hand, by \eqref{E:normal-hy-1}, similarly as \eqref{E:lower-bdd-1-1} we have
  \begin{equation}\label{E:normal-hy-4}
 \begin{split}
   &z(\varphi(t,\cdot;u_{2},g)-\varphi(t,\cdot;u_1,g))\\
   &=z(\varphi(t,\cdot;u_{12},g)-\varphi(t,\cdot;u_1,g))\geq N_u+2,\quad \text{for}\ t\ll -1.
 \end{split}
 \end{equation}
By \eqref{E:normal-hy-3}-\eqref{E:normal-hy-4}, we obtain a contradiction to Lemma \ref{constant-on-twosets}. Thus, case (b1) is impossible.
\vskip 2mm

For case (b2): Again, by Theorem \ref{structure}(iii), we may assume that $\omega(u_{12},g)\cap M_1\neq \emptyset$ and $\alpha(u_{12},g)\cap M_2\neq \emptyset$. This then implies that $\a(u_{12},g)\cap M_1=\emptyset$ and $\omega(u_{12},g)\cap M_2=\emptyset$ (Otherwise, $\{(u_1,g),(u_{12},g)\}$ or $\{(u_2,g),(u_{12},g)\}$ will  form a two sided proximal pair. Note that $M_1,M_2$ are spatially-homogeneous. This contradicts Lemma \ref{norma-hyperb2}(ii).) Similarly as \eqref{E:approach-each} and the assertion in the proof of Theorem \ref{hyperbolic0},
we obtain that
\begin{equation}\label{E:double-homo-1}
\Pi^{t}(u_{12},g)-\Pi^{t}(u_{1},g)\to 0\text{ as }t\to\infty
 \end{equation}
and
\begin{equation}\label{E:double-homo-2}
 \Pi^{t}(u_{12},g)-\Pi^{t}(u_{2},g)\to 0\text{ as }t\to-\infty.
 \end{equation}
Consequently, it follows from the hyperbolicity of $M_1,M_2$ that  $\Pi^t(u_{12},g)\in \tilde M^s_1((\Pi^t(u_{1},g),\delta^*)$ for $t\gg 1$ and $\Pi^t(u_{12},g)\in \tilde M^u_2((\Pi^t(u_{2},g),\delta^*)$ for $t\ll- 1$, where $\tilde M_i^u(\omega,\delta^*),\tilde M_i^s(\omega,\delta^*), i=1,2,$ denote respectively the local unstable and stable manifolds of $\omega\in M_i$ with respect to the Sacker-Sell spectrum $\sigma(M_i)$.

Then, by Lemma \ref{homogeneous}(b), Lemma \ref{zerocenter}(1) and Lemma \ref{difference-lapnumber}(a), one has
 \begin{equation}\label{F:lower-bdd-1}
   z(\varphi(t,\cdot;u_{1},g)-\varphi(t,\cdot;u_{12},g))\geq  N_u, \quad t\in \mathbb{R}^1.
 \end{equation}

Recall that $\Pi^{t}(u_{12},g)-\Pi^{t}(u_{2},g)\to 0$ as $t\to -\infty$. Then it follows from \eqref{F:lower-bdd-1}, similarly as in \eqref{E:lower-bdd-1-1}, that we have
 \begin{equation*}
 \begin{split}
   &z(\varphi(t,\cdot;u_{1},g)-\varphi(t,\cdot;u_2,g))\\
   =&z(\varphi(t,\cdot;u_{1},g)-\varphi(t,\cdot;u_{12},g))\geq N_u>0,\quad \text{for}\ t\,\text{ sufficiently negative},
 \end{split}
 \end{equation*}
which contradicts $z(\varphi(t,\cdot;u_{1},g)-\varphi(t,\cdot;u_2,g))=0$ (since both $u_1$ and $u_2$ are spatially-homogeneous).
Thus, case (b2) should not happen.

\vskip 3mm

For case (b3): We may assume without loss of generality that $M_1$ is spatially-homogeneous and $M_2$ is spatially-inhomogeneous.  By Theorem \ref{structure}(iii) again, we may also assume that $\omega(u_{12},g)\cap M_1\neq \emptyset$ and $\alpha(u_{12},g)\cap M_2\neq \emptyset$. So, by Lemma \ref{norma-hyperb2}(i)-(ii) and the same argument in case (b1) and (b2), we have
$\a(u_{12},g)\cap M_1=\emptyset$ and $\omega(u_{12},g)\cap M_2=\emptyset$. Therefore, one can further obtain \eqref{E:double-homo-1} and \eqref{E:double-homo-2} similarly.

When $\dim V^u(\Omega)$ is odd (resp.  $\dim V^u(\Omega)$ is even), it follows from \eqref{E:double-homo-1} and \eqref{E:cs=unstable} that
\begin{equation}\label{E:homo-inhomo-2sided-1}
\Pi^{t}(u_{12},g)\in M^{cs}(\Pi^t(u_1,g),\delta^*)=\tilde M^{s}(\Pi^t(u_1,g),\delta^*), \,\text{ for }t\gg 1;
\end{equation}and from \eqref{E:double-homo-2}, one has $$\Pi^{t}(u_{12},g)\in M^{cu}(\Pi^t(u_2,g),\delta^*), \,\text{ for }t\ll -1.$$
Recall that $M_2$ is inhomogeneous. Together with Remark \ref{stable-leaf} and Lemma \ref{homogeneous}(a), one can repeat the similar argument in the proof of the assertion after \eqref{E:approach-each} to obtain
 $$\Pi^{t}(u_{12},g)\in M^{u}(\Pi^t(u_2,g),\delta^*), \,\text{ for }t\ll -1.$$

So, by Lemma \ref{zerocenter}(2) (resp. Lemma \ref{zerocenter}(3)), we have
\begin{equation}\label{G:normal-hy-2}
  z(\varphi(t,\cdot;u_{12},g)-\varphi(t,\cdot;u_{2},g))\le
   \left\{
\begin{split}
 &N_u-2,\,\,\,\,\text{ if }{\rm dim}V^u(\Omega)\text{ is odd,}\\
 &N_u,\,\quad\quad\text{ if }{\rm dim}V^u(\Omega)\text{ is even,}
\end{split}\right.
\end{equation}
for all $t\in \mathbb{R}.$

On the other hand, by letting $E=M_1$ in Lemma \ref{zerocenter} and noticing that $\dim \tilde M^u_1$
 is always odd, we deduce from by Lemma \ref{homogeneous}(b), Lemma \ref{zerocenter}(1) and \eqref{E:homo-inhomo-2sided-1} that
\begin{equation}\label{G:normal-hy-1}
  z(\varphi(t,\cdot;u_{12},g)-\varphi(t,\cdot;u_{1},g))\ge\dim \tilde M^u_1+1,\,\quad \text{for all}\ t\in \mathbb{R}.
\end{equation}
Since
\begin{equation*}\label{G:normal-hy-1-1}
  \dim \tilde M^u_1=\left\{
\begin{split}
 &{\rm dim}V^u(\Omega),\,\,\,\,\quad\,\,\text{ if }{\rm dim}V^u(\Omega)\text{ is odd,}\\
 &{\rm dim}V^u(\Omega)+1,\,\,\text{ if }{\rm dim}V^u(\Omega)\text{ is even,}
\end{split}\right.
\end{equation*}
it follows from \eqref{L:lap-classify} that $\dim \tilde M^u_1=N_u-1$, if ${\rm dim}V^u(\Omega)\text{ is odd;}$ and $\dim \tilde M^u_1=N_u+1$, if ${\rm dim}V^u(\Omega)\text{ is odd.}$ Therefore,  \eqref{G:normal-hy-1} entails that
\begin{equation}\label{G:normal-hy-2-2}
  z(\varphi(t,\cdot;u_{12},g)-\varphi(t,\cdot;u_{2},g))\ge
   \left\{
\begin{split}
 &N_u,\,\,\quad\quad\text{ if }{\rm dim}V^u(\Omega)\text{ is odd,}\\
 &N_u+2,\,\,\,\text{ if }{\rm dim}V^u(\Omega)\text{ is even,}
\end{split}\right.
\end{equation}
for all $t\in \mathbb{R}.$

For \eqref{G:normal-hy-2}, similarly as \eqref{E:lower-bdd-1-1} we utilize \eqref{E:double-homo-1} to obtain that
\begin{equation}\label{G:normal-hy-4}
 \begin{split}
   &z(\varphi(t,\cdot;u_{2},g)-\varphi(t,\cdot;u_1,g))\\
   &=z(\varphi(t,\cdot;u_{12},g)-\varphi(t,\cdot;u_2,g))\leq N_u-2\,\, (\text{resp. } N_u),\quad \text{for}\ t\gg 1.
 \end{split}
 \end{equation}
Similarly, we use \eqref{G:normal-hy-2-2} and \eqref{E:double-homo-2} to obtain that
\begin{equation}\label{G:normal-hy-3}
 \begin{split}
   &z(\varphi(t,\cdot;u_{2},g)-\varphi(t,\cdot;u_1,g))\\
   &=z(\varphi(t,\cdot;u_{12},g)-\varphi(t,\cdot;u_1,g))\geq N_u\,\, (\text{resp. } N_u+2),\quad \text{for}\ t\ll -1.
 \end{split}
 \end{equation}
However,  by Lemma \ref{constant-on-twosets}, it is already known that
\begin{equation*}
z(\varphi(t,\cdot;u_{2},g)-\varphi(t,\cdot;u_1,g))=\text{constant},\quad \text{for all}\ t\in \mathbb{R},
\end{equation*}
a contradiction to \eqref{G:normal-hy-4}-\eqref{G:normal-hy-3}. Thus, case (b3) cannot happen.

Up to now, we have shown that none of case (b1)-(b3) can happen, which implies that Theorem \ref{structure}(iii) can not happen at all. Thus, we have proved the {\it claim that the $\omega$-limit set $\O$ itself is minimal}, which completes our proof.
\end{proof}

\section{Examples}
In this section, we will present some examples related to Theorem 3.1 and Theorem 5.1.
As for Theorem \ref{structure} (concerning with the structure of $\Omega$), we will give some example to show that case (ii)-(iii) in Theorem \ref{structure} indeed may occur. For Theorem \ref{norma-hyper} (concerning with the structure of $\Omega$ with $\dim V^c(\Omega)=1$), we will also present an example to show that,  if $\dim V^c(\Omega)\ge 2$, then the conclusion in Theorem \ref{norma-hyper} cannot hold anymore.

 \vbox{}
  Example 6.1. Consider the following parabolic equation (motivated by \cite{ShenYi-2}):
\begin{equation}\label{examp1}
  u_t=u_{xx}+(f(t)-\lambda)u,\,\,t>0,\,x\in S^{1}=\mathbb{R}/2\pi \mathbb{Z},
\end{equation}
where $f(t)=-\sum_{k=1}^{\infty}2^{-k}\pi \mathrm{sin}(2^{-k}\pi t)$ is an almost periodic function, and $\lambda$ is an eigenvalue of $\frac{\partial^2}{\partial x^2}: H^2(S^1)\to L^2(S^1)$.

The skew-product semiflow $\Pi^t$ on $X\times H(f)$ is
\begin{equation}
  \Pi^t(u,g)=(\varphi(t,\cdot; u,g),g\cdot t),
\end{equation}
where $X$ is the fractional power space defined in the introduction. Let $u_0$ be an eigenfunction of $\frac{\partial^2}{\partial x^2}: H^2(S^1)\to L^2(S^1)$ corresponding to $\lambda$. Then $\varphi(t,\cdot; u_0,f)=e^{\int_{0}^t f(s) ds}u_0$ is a solution of \eqref{examp1}.

Following the discussion in \cite{Sacker77,ShenYi-2}, the function $\psi(t)=e^{\int_{0}^t f(s) ds}$ satisfies the following properties:

(a) $\psi(t)$ is bounded for $t\geq 0$;

(b) There exists $t_n\to\infty$ such that $\psi(t_n)\to 0$ as $n\to\infty$, and $\psi(2^n)\geq e^{-2\pi-2}$ for $n=1,2,\cdots$;

(c) For any sequence $t_n\to\infty$ such that $\lim_{n\to\infty}\psi(t+t_n)=\psi^*(t)$ exists, $\psi^*(t)$ is not almost periodic if it is nonzero.

By virtue of (b), $\omega(u_0,f)$ is not minimal and $\{0\}\times H(f)$ is the only minimal set in $\omega(u_0,f)$. Moreover, $\omega(u_0,f)$ is an almost $1$-cover of $H(f)$ (see \cite[p.396]{ShenYi-2}).

\vskip 2mm
The linearized variational equation of \eqref{examp1} at $(u,g)\in \omega(u_0,f)$ is
\begin{equation}\label{linear-variation}
  v_t=v_{xx}+(g(t)-\lambda)v.
\end{equation}
Clearly, $e^{-k^2 t+\int_0^t (g(s)-\lambda)ds}\mathrm{sin}(kx)$ and $e^{-k^2 t+\int_0^t (g(s)-\lambda)ds}\mathrm{cos}(kx)$ are solutions of \eqref{linear-variation} for $k\in \mathbb{N}$). Moreover, these functions form a Floquet basis of \eqref{linear-variation}, and the associated Floquet spaces are $W_0=\mathrm{span}\{e^{\int_0^t (g(s)-\lambda)ds}\}$ and
\begin{equation}\label{floquet-1}
W_k=\mathrm{span}\{e^{-k^2 t+\int_0^t (g(s)-\lambda)ds}\mathrm{sin}(kx), e^{-k^2 t+\int_0^t (g(s)-\lambda)ds}\mathrm{cos}(kx)\},\,
\,k=1,2,\cdots.
\end{equation} Such Floquet spaces are exponentially separated. Note also that
$\lim_{\abs{t}\to \infty}\frac{1}{t}\int_{0}^tg(s)ds=0$ (see \cite{Fink,Shen1995114}), it follows that the Sacker-Sell spectrum $\sigma(\omega(u_0,f))=\{-\lambda-k^2:k\in \mathbb{N}\}$.

\vskip 3mm
\textbf{Case (i)}: Let $\lambda=0$ (which is the first eigenvalue of $\frac{\partial^2}{\partial x^2}: H^2(S^1)\to L^2(S^1)$) and the corresponding eigenfunction $u_0\equiv 1$ is spatially-homogeneous). Then, it is clear that $\omega(u_0,f)$ is spatially-homogeneous and $\sigma(\omega(u_0,f))=\{0,-1,\cdots,-k^2,\cdots\}$ (Thus, the upper Lyapunov exponent with respect to \eqref{linear-variation} is $0$. Hence, for this case, one may also call that $\omega(u_0,f)$ is linearly stable by following \cite[p.36, Definition II.4.3]{Shen1998}). Moreover, one has $\dim V^u(\omega(u_0,f))=0$ and $\dim V^c(\omega(u_0,f))=1$. As we mentioned in the previous paragraph, $\omega(u_0,f)$ is not minimal and $\{0\}\times H(f)$ is the only minimal in $\omega(u_0,f)$. Therefore, it belongs to Theorem \ref{structure}(ii). Furthermore, it is also an example indicating that $\omega(u_0,f)$ need not be minimal even if it belongs to Theorem \ref{norma-hyper}(ii).

\vskip 2mm
\textbf{Case (ii)}: Let $\lambda=-1$ (which is the second eigenvalue of $\frac{\partial^2}{\partial x^2}: H^2(S^1)\to L^2(S^1)$) and choose the corresponding eigenfunction $u_0=\mathrm{sin}x$). Then, the corresponding solution is $\varphi(t,x,u_0,g)=e^{\int_0^tf(s)ds} \sin x$. So, $u$ is spatially-inhomogeneous for all $(u,g)\in \omega(u_0,f)\setminus\{0\}\times H(f)$.
Moreover, we have $\sigma(\omega(u_0,f))=\{1,0,\cdots,1-k^2,\cdots\}$. Together with \eqref{floquet-1}, we also obtain that $\dim V^c(\omega(u_0,f))=2$.
Note also that $\omega(u_0,f)$ is not minimal. Thus, such an example shows that,  if $\dim V^c(\Omega)\ge 2$, then neither (i) nor (ii) in Theorem \ref{norma-hyper} need hold.

\vbox{}
\textbf{Example 6.2}. Consider
\begin{equation}
\label{neq-example}
u_t=u_{xx}-(a(t)\cos u+b(t)\sin u)\sin u,\,\, t>0,\,\, x\in S^1=\mathbb{R}/2\pi \mathbb{Z},
\end{equation}
where $f(t)=(a(t),b(t))$ is almost periodic and satisfies that the ODE
\begin{equation}
\label{neq-ode}
y^{'}=a(t)y+b(t)
\end{equation}
admits no almost periodic solution and the solution $y_0(t)$ of \eqref{neq-ode} with $y_0(0)=0$ is bounded (see  \cite{Johnson1981} for the existence of such $a(t)$ and $b(t)$). Then there are $(u_0,g_0)\in X\times H(f)$ such that
$\omega(u_0,g_0)$  contains two minimal sets (see \cite{Shen1995114} for the detail). So, this example indicates that Theorem \ref{structure}(ii) can indeed occur.

\section*{Acknowledgements}
The authors are greatly indebted to an anonymous referee for very careful reading and providing lots of very inspiring and helpful comments and suggestions which led to much improvement of two earlier versions of this paper.

\end{document}